\documentclass[11pt,letterpaper]{amsart}
\usepackage{amssymb}
\usepackage{amsfonts}
\usepackage{amsmath}
\usepackage{graphicx}
\usepackage{xcolor}
\usepackage{amsmath}
\usepackage{mathrsfs}
\usepackage{stmaryrd}
\usepackage{epsfig,color}
\usepackage{blindtext}
\usepackage{enumerate}
\usepackage[citecolor=blue,colorlinks]{hyperref}

\usepackage{url}
\usepackage{bbm}
\usepackage{bm}   
\DeclareGraphicsExtensions{.pdf,.jpeg,.png}
\usepackage{epstopdf}
\usepackage{cancel} 
\usepackage[normalem]{ulem} 
\usepackage{verbatim} 
\usepackage{enumitem} 
\usepackage{soul}
\usepackage{color}

\usepackage{footnote}
\usepackage[msc-links, lite]{amsrefs}

\usepackage{geometry}
\newtheorem{theorem}{Theorem}[section]

\newtheorem{proposition}[theorem]{Proposition}
\newtheorem{lemma}[theorem]{Lemma}
\newtheorem{corollary}[theorem]{Corollary}
\newtheorem{question}[theorem]{Question}

\newtheorem*{acknowledgements}{Acknowledgements}
\theoremstyle{definition}
\newtheorem{definition}[theorem]{Definition}
\newtheorem{remark}[theorem]{Remark}
\numberwithin{equation}{section}

\newcommand{\mb}{\mathbb}
\newcommand{\mc}{\mathcal}

\newcommand{\eps}{\varepsilon}

\newcommand{\mr}{\mathrm}

\DeclareMathOperator{\Ric}{Ric}

\newcommand{\hess}{{\mathrm{Hess}}}
\newcommand{\volg}{\mathrm{vol}_g}

\newcommand{\C}{\mathbb{C}}
\newcommand{\N}{\mathbb{N}}

\newcommand{\R}{\mathbb{R}}
\renewcommand{\subset}{\subseteq}
\newcommand{\defeq}{\mathrel{\mathop:}=}

\newcommand{\haus}{\mathcal{H}}

\newcommand{\dist}{\mathsf{d}}

\newcommand{\meas}{\mathfrak{m}}

\newcommand{\di}{\mathop{}\!\mathrm{d}}

\DeclareMathOperator{\RCD}{RCD}
\DeclareMathOperator{\ncRCD}{ncRCD}

\def\Xint#1{\mathchoice
{\XXint\displaystyle\textstyle{#1}}%
{\XXint\textstyle\scriptstyle{#1}}%
{\XXint\scriptstyle\scriptscriptstyle{#1}}%
{\XXint\scriptscriptstyle\scriptscriptstyle{#1}}%
\!\int}
\def\XXint#1#2#3{{\setbox0=\hbox{$#1{#2#3}{\int}$ }
\vcenter{\hbox{$#2#3$ }}\kern-.6\wd0}}

\def\dashint{\Xint-}

\title{On the geometry at infinity of manifolds with linear volume growth and nonnegative Ricci curvature}

\author{Xingyu Zhu \textsuperscript{1}}\thanks{  \textsuperscript{1} Institute for Applied Mathematics, University of Bonn, Endenicher Allee 60, 53115 Bonn, Germany. Email: {zhu@iam.uni-bonn.de}}

 \subjclass[2020]{Primary 53C21, 53C23.\\
\indent key words: Ricci curvature, splitting, linear volume growth, isoperimetric sets}

\begin{document}
\begin{abstract}
We prove that an open noncollapsed manifold with nonnegative Ricci curvature and linear volume growth always splits off a line at infinity. This completes the final step to prove the existence of isoperimetric sets for given large volumes in the above setting. We also find that under our assumptions, the diameters of the level sets of any Busemann function are uniformly bounded as opposed to a classical result stating that they can have sublinear growth when the end is collapsing. Moreover, some equivalent characterizations of linear volume growth are given. Finally, we construct an example to show that for manifolds in our setting, although their limit spaces at infinity are always cylinders, the cross sections can be nonhomeomorphic.

\end{abstract}
\maketitle

\section{Introduction}
Recently there has been an interest in studying the structure at infinity on manifolds with nonnegative Ricci curvature, meaning the pointed Gromov-Hausdorff (pGH in short) limits of the form $(M,g,p_i)$ for an open manifold and a sequence of points diverging to infinity (see Definition \ref{def:endsvolinfty}). The motivations come from several different fields. For instance, one motivation is to understand positive scalar curvature coupled with nonnegative Ricci curvature at large scales \cites{Zhu_Geometryofpsc, WZZZ_PSC_RLS, ZhuZhu23,zhu2023twodimension} and another is to examine the existence of isoperimetric sets in a noncompact manifold \cites{isoregion, antonelli2023isoperimetric} with nonnegative Ricci curvature and Euclidean or linear volume growth. In both cases, a common interest is to see if any limit space at infinity (see Definition \ref{def:endsvolinfty}) splits off many lines, because splitting at infinity simplifies the structure of the limit spaces under our consideration, making it possible to apply some powerful techniques. 

In light of this demand, we take the very first step to study in particular the line splitting at infinity for noncollapsed manifolds with linear volume growth. We introduce these notions as follows. 

\begin{definition}\label{def:endsvolinfty}
    Let $(M,g)$ be an open manifold. We say that $M$ is noncollapsed if 
    \begin{equation}\label{eq:ncends}
       \text{there exists $v>0$ such that}\ \volg(B_1(x))>v,\  \forall x\in M.
    \end{equation}
    We say that $(M,g)$ has linear volume growth if for some point (hence all points) $p\in M$,
    \begin{equation}\label{eq:linearvol}
      \text{there exists $V>0$ such that}\   \limsup_{r\to\infty}\frac{\volg(B_r(p))}{r}=V.
    \end{equation}
    We say that a sequence of points $\{p_i\}_{i\in \N^+}$ diverges to infinity (or it is a diverging sequence) if for any fixed point $q\in M$, $\dist(q,p_i)\to \infty$ as $i\to \infty$. We say a Ricci limit space is a limit space at infinity of $M$ if it is obtained as a pGH limit of sequence of the form $(M,g, p_i)$ for a sequence $\{p_i\}_{i\in \N^+}$ diverging to infinity. We say a metric space is a cylinder if it is a metric product in the form of $\R\times K$ with $K$ being a compact metric space.

\end{definition} 
Examples of manifolds with linear volume growth that are not necessarily a product of $\R$ and a compact manifold include $3$-manifolds with nonnegative Ricci curvature and uniformly positive scalar curvature as confirmed by \cites{OLS23, MW_geometryofpsc}.

Our first result is about splitting at infinity with the noncollapsed assumption. 

\begin{theorem}\label{thm:splitting}
    Let $(M,g)$ be an open $n$-manifold with $\Ric_g\ge 0$. Suppose $M$ is noncollapsed \eqref{eq:ncends}, and has linear volume growth with a volume ratio upper bound $V>0$ as in \eqref{eq:linearvol}. Then for any sequence $\{p_i\}_{i\in \N^+}$ diverging to infinity, all possible pGH limit spaces of the sequence $(M,g,p_i)$ splits as $\R\times K$, where $K$ is a compact $\ncRCD(0,n-1)$ space with $\haus^{n-1}(K)\le nV$. Moreover, all possible $K$ have the same $\haus^{n-1}$ volume. 
\end{theorem}

\begin{remark}\label{rmk:ncends}
   There are extensive studies of manifolds with linear volume growth by Sormani \cites{SormaniMiniVol,SormaniSublinear} without the noncollapsed assumption. The motivation for considering the noncollapsed assumption \eqref{eq:ncends} is also natural. In her pioneering work, Sormani \cite{SormaniMiniVol}*{Example 26} constructed a $4$-manifold with nonnegative Ricci curvature and linear volume growth, on which there is a Busemann function whose level sets have logarithmic diameter growth and finite codimension $1$ volume. In this example, the limit space at infinity is $\R^3$, it does split off an $\R$-factor but the other factor is $\R^2$ which is noncompact as opposed to Theorem \ref{thm:splitting}. 
    
    Intuitively, for the diameter to tend to infinity while the volume stays bounded, there should be some ``direction'' that is shrinking, leading to a collapsing end. Therefore it is believable that $\eqref{eq:ncends}$ can prevent the diameter of the level sets of a Busemann function from tending to infinity. Indeed, it will be shown in proposition \ref{prop:uniBddDiam}.
\end{remark}

    \begin{remark}
        The cross sections $K$ in the limit space at infinity can indeed be nonisometric, as shown by an example of Sormani \cite{SormaniMiniVol}*{Example 27}. In this example the limit space splits as $\R\times \mb S^3$ but for different diverging sequence of points the $\mb S^3$-factor can carry two different metrics, and both metrics give the same volume to $\mb S^3$. We will see in Theorem \ref{thm:3dnoniso} that such an example with nonisometric $K$ can have dimension $3$ and in Theorem \ref{thm:nonhomeo} that the cross sections $K$ can be nonheomeomorphic.    
    \end{remark}

   The following example of Kasue-Washio \cite{KasueExample}*{p.913-914} helps illustrate why the above Theorem is not obvious, Specifically, it shows that with nonnegative Ricci curvature, not every diverging sequence of points produces a limit space at infinity that splits. So at least the linear volume growth condition plays a key role. The example says there is a Riemannian metric $g$ on $\R^4=\R\times\R\times \mb S^2$ of the form $g=f^2(r)\di t^2+\di r^2 +\eta(r)g_{\mb S^2}$ with nonnegative Ricci curvature. From its expression $g$ is translation invariant along $\partial_t$ direction, but $(\R^4,g)$ does not split. When taking a sequence of points diverging to infinity along $\partial_t$ direction we get a limit space isometric to $(\R^4,g)$, the manifold we start with, which does not split. See also some discussions in \cite{isoregion}*{Section 1.4}. In this example $(\R^4, g)$ has Euclidean volume growth instead of linear volume growth. In the seemingly simple and restrictive case where $M$ has nonnegative Ricci curvature and linear volume growth, to the author's best knowledge the following splitting problem remains open, i.e., when the noncollapsed assumption in Theorem \ref{thm:splitting} is removed.

\begin{question}
	Given $(M,g)$ with $\Ric_g\ge 0$ and linear volume growth \eqref{eq:linearvol}, is it true that for every sequence $\{p_i\}_{i\in \N^+}$ diverging to infinity, the pGH limit of $(M,g, p_i)$ splits off a line?
\end{question} 
    
 Unlike the rather complicated situation for Ricci curvature, in the case of manifolds with nonnegative sectional curvature, the picture for splitting is much simpler. In fact, the following two statements holds.
 \begin{enumerate}
     \item For a manifold with nonnegative sectional curvature, any limit space at infinity always splits off a line without further assumptions \cite{antonelli2023isoperimetric}*{Lemma 2.29}, \cite{zhu2023twodimension}*{Corollary 4.3}.
     \item For a manifold with nonnegative sectional curvature, itself splits if and only if its asymptotic cone splits. See for example a proof in \cite{isoregion}*{Theorem 4.6} and the references therein.
 \end{enumerate} 
 The very reason is the monotonicity of angles, which allows us to pass large scale information to smaller scales or vice versa. It is then pointed out by the author \cite{zhu2023twodimension}*{Lemma 4.2} that the above two statements can be unified.

 The above discussion gives a nice relation between asymptotic cones and limit spaces at infinity. In general we do not have such a relation for nonnegative Ricci curvature. A counterexample is again the one of Kasue-Washio's we discussed in the introduction \cite{KasueExample}, \cite{isoregion}*{section 1.4}. In this example, $(\R^4,g)$ has a unique asymptotic cone and it splits, but $(\R^4,g)$ itself does not. In the very special case of linear volume growth and being noncollapsed, Theorem \ref{cor:splitting} provides some similar relation. 

Nevertheless the idea of studying the relation between asymptotic cones and limit spaces at infinity still leads us to the following interesting observations or analogs. Note that a manifold with nonegative Ricci curvature and Euclidean volume growth is noncollapsed by Bishop-Gromov inequality.
\begin{itemize}
    \item For a manifold with Euclidean (maximal) volume growth, its asymptotic cones are metric cones and the cross sections of the asymptotic cones always have the same volume \cite{Cheeger-Colding97I}.
    \item For a noncollapsed manifold with linear (minimal) volume growth, its limit at infinity are cylinders and the cross sections of the cylinders always have the same volume.
	\item There exists a $5$-manifold with Euclidean (maximal) volume growth and its asymptotic cones can be a cone over either $\C P^2\sharp\overline{\C P^2}$ or $\mb S^4$, \cite{CN11}*{Example II}.   
	\item There exists a noncollapsed $5$-manifold with linear (minimal) volume growth and its limit spaces at infinity can be a product of $\R$ with either $\C P^2\sharp\overline{\C P^2}$ or $\mb S^4$.	
\end{itemize}   

After a in depth analysis, we can establish the following list of equivalence conditions, giving a rather complete description of the geometry of linear volume growth. Also, it partially extends \cite{antonelli2023isoperimetric}*{Corollary 3.11} to nonnegative Ricci curvature and partially answers the questions in \cite{antonelli2023isoperimetric}*{Problem 1.2}. However, there are remaining questions. For example we cannot show that if for some divergent sequence of points the corresponding limit space at infinity splits then for arbitrary divergent sequence of points the corresponding limit space at infinity also splits. This implication, if true, will link the isometric profile bound to the linear volume growth. See Remark \ref{rmk:isoprofile} for details.

\begin{theorem}\label{thm:equi}
        Let $(M,g)$ be an open $n$-manifold with $\Ric_g\ge 0$ such that $M$ has noncollapsed ends \eqref{eq:ncends}. The following statements are equivalent.
        \begin{enumerate}
            \item\label{item:linearvol} $M$ has linear volume growth.
            \item\label{item:uniclose1} Either $M$ splits off a line or for any ray $\gamma$ there exists $D>0$ such that $\dist_g(p,\gamma)<D$ for every $p\in M$.
            \item\label{item:uniclose2} Either $M$ splits off a line or there exists a ray $\gamma$ and $D>0$ such that $\dist_g(p,\gamma)<D$ for every $p\in M$.
            \item\label{item:splitting1} There exists a ray $\gamma$ such that for any sequence $t_i\nearrow\infty$, up to subsequence $(M,\gamma(t_i))$ pGH converges to $(\R\times K, (0,k_0))$ where $K$ is a compact $\ncRCD(0,n-1)$ space possibly depending on the sequence and $k_0\in K$. 
            \item\label{item:splitting2} For any sequence $\{p_i\}$ diverging to infinity, $(M,p_i)$ pGH converges to $(\R\times K, (0,k_0))$, for some $K$ that is a compact $\ncRCD(0,n-1)$ space and $k_0\in K$.
     \end{enumerate}        
    \end{theorem}

To close this section, we point out the technical difficulties we overcome in this note. A guiding principle is to use Cheeger--Colding's almost splitting theorem to replace the arguments in \cite{antonelli2023isoperimetric}*{Section 3}, which exploit Alexandrov geometry. The overall idea for showing the equivalence in Theorem \ref{thm:equi} is essentially the same, so we will emphasize our idea of proof of Theorem \ref{thm:splitting}. Our proof is based on a special case of the almost splitting theorem, when the almost splitting happens along a ray. We describe it below. 

Let $(M,g)$ be an open $n$-manifold with $\Ric_g\ge0$, and $\gamma:[0,\infty)\to M$ be a ray. The Busemann function $b_\gamma:M\to\R$ associated to $\gamma$, which will be introduced more formally in the next section, is defined as
\begin{equation*}
	b_\gamma(p)=\lim_{t\to\infty} t-\dist_g(p,\gamma(t)).
\end{equation*}
Recall that in Cheeger--Gromoll splitting theorem if $M$ actually contains a line $\ell:\R\to M$, then $M$ splits exactly as a product of $\R$ and a level set of the Busemann function $b_\ell^{-1}(0)$ associated to $\ell|_{[0,\infty)}$. In the setting of the almost splitting theorem, let $\gamma:[0,\infty)\to M$ be a ray and $t_i\nearrow\infty$. Consider the Gromov-Hausdorff limit of a convergent subsequence of $(M,g,\gamma(t_i))$. Viewing at the point $\gamma(t_i)$, $\gamma|_{[t_i,\infty]}$ is still a ray and it extends backwards to $\gamma|_{[0,t_i]}$, so we almost have a line. As an analog of Cheeger--Gromoll splitting theorem, it is a folklore result that the almost splitting theorem applied to $B_R(\gamma(t_i))$ gives that $B_R(\gamma(t_i))$ is Gromov-Hausdorff close to a ball of the same radius in $\R\times b_\gamma^{-1}(t_i)$. However, this is not exactly the statement of the almost splitting theorem. Instead of working with the actual Busemann function, the almost splitting theorem works with the harmonic replacement of it, see Theorem \ref{thm:AlmostSplitting}. The most technical part of this paper is to clarify in precise sense how this folklore result is true. This will be the theme of Section \ref{sec:ALS}. The rest will be straightforward applications, as we will see in Section \ref{sec:infinity}.  

\begin{acknowledgements}
    The author would like to thank Gioacchino Antonelli and Marco Pozzetta for sending their paper \cite{antonelli2023isoperimetric} and Gioacchino Antonelli's master thesis.  \cite{antonelli2023isoperimetric} inspires the author to write the present manuscript. The author would also like to thank Shouhei Honda, Jiayin Pan, and Daniele Semola for reading a preliminary version of the manuscript and providing valuable feedbacks, Zhu Ye for spotting a mistake in the previous version, Igor Belegradek and Zetian Yan for some discussions of the constructions in section \ref{sec:example}, Sergio Zamora for helpful suggestions on Proposition \ref{prop:LevelDiam}. The author also appreciate the anonymous referee for patiently correcting several typos and the suggestions that greatly improve the readability.  
\end{acknowledgements}

\section{Busemann function and the almost splitting Theorem}\label{sec:ALS}
 \subsection{Busemann functions}
 Busemann functions play a fundamental role in the study of noncompact manifolds with nonnegative Ricci curvature. Some fine properties of them are studied by Sormani \cites{SormaniMiniVol, SormaniSublinear} in the case of linear volume growth. In this section we will review some of the results from Sormani's work and establish some direct implications. Since we are mostly dealing with metric geometry problems, we make a convention that all geodesics mentioned in this note are minimizing geodesics. 
 
 We start by giving the definition of a Busemann function. Let $(M,g)$ be an open $n$-manifold. We call a curve $\gamma:[0,\infty)\to M$ a ray on $M$ if $\gamma$ minimizes the distance globally, i.e. $$\dist_g(\gamma(s),\gamma(t))=|t-s|$$ for $s,t\ge 0$. Given a ray $\gamma$, the Busemann function $b_\gamma$ associated to $\gamma$ is defined by a monotone increasing limit
\begin{equation*}
	b_\gamma(p)=\lim_{t\to\infty} t-\dist_g(p,\gamma(t)).
\end{equation*}
It is easy to see $ t-\dist_g(p,\gamma(t))\le \dist_g(p,\gamma(0))$ so the limit always exists. $b_\gamma$ is a $1$-Lipschitz function by triangle inequality and when $\Ric_g\ge0$ it is a subharmonic function in the sense of distribution by the Laplacian comparison theorem.  

The following lemma comes handy and will be repeatedly used in the sequel.

\begin{lemma}[\cite{SormaniMiniVol}*{Lemma 6}]\label{lem:Sormani}
Let $\gamma$ be a ray and $b_\gamma$ be its associated Busemann function. For any real number $R$, we have for any $x\in b_\gamma^{-1}((-\infty, R])$ that
\[
\dist_g(x,b_\gamma^{-1}(R))= R-b_\gamma(x).
\]	
In particular, if $\mr{diam}(\{b_\gamma=R\})<\infty$ and $r<R$, then $\mr{diam}(\{b_\gamma=r\})\le \mr{diam}(\{b_\gamma=R\})+2(R-r)$.
\end{lemma}

Lemma \ref{lem:Sormani} is simple but powerful. We will make several uses of the following consequence of it both in this and the next section. 

\begin{corollary}\label{cor:LevelClose}
    Let $(M,g)$ be an $n$-manifold with $\Ric_g\ge 0$ and $\gamma$ be a ray. Fix $p\in M$. Let $t_p\in \R$ and $R\ge 0$ be such that such $\dist_g(\gamma(t_p), p)\le R$, and let $t_p'\defeq b_\gamma(p)$. Then 
    \begin{equation}
        |t_p-t_p'|\le R.
    \end{equation}
\end{corollary}

\begin{proof}
     if $t_p'\ge t_p$, then since $t_p'=b_\gamma(p)\le \dist_g (p,\gamma(0))$ and $$t_p=\dist_g(\gamma(0),\gamma(t_p))\ge \dist_g(\gamma(0),p)-\dist_g(p,\gamma(t_p))>\dist_g(\gamma(0),p)-R\ge t_p'-R,$$ we see immediately $0\le t_p'-t_p\le R$. On the other hand if $t_p\ge t_p'$, then by Lemma \ref{lem:Sormani} we have $t_p-t_p'=\dist_g(p,b_\gamma^{-1}(t_p))\le \dist_g(p,\gamma(t_p))\le R$, as claimed. 
\end{proof}

We also mention a result that directly follows from Sormani's work \cite{SormaniMiniVol}.

\begin{lemma}\label{lem:bddBuse}
     Let $(M,g)$ be an open $n$-manifold with $\Ric_g\ge 0$. Let $\gamma$ be a ray in $M$ and $b_\gamma$ is the associated Busemann function. Suppose $M$ has linear volume growth \eqref{eq:linearvol}. Then either $M$ splits or $M$ has only one end and for each $T\in \R$, $\{b_\gamma\le T\}$ is compact. 
\end{lemma}

\begin{proof}
    If $M$ does not split, then it follows from \cite{SormaniMiniVol}*{Corollary 23} $m\defeq\inf_{x\in M}b_\gamma(x)>-\infty$ exists. Assume $T\ge m$ otherwise there is nothing to prove. Take $x\in \{b_\gamma\le T\}$. Notice that by Lemma \ref{lem:Sormani}, $ \dist_g(x, \{b_\gamma=T\})=T-b_\gamma(x)\le T-m$, and by \cite{SormaniMiniVol}*{Theorem 19} that $\{b_\gamma=T\}$ is compact, we can estimate that 
    \[
    \dist_g(x,\gamma(T))\le \dist_g(x, \{b_\gamma=T\})+\mr{diam}(\{b_\gamma=T\})\le T-m+\mr{diam}(\{b_\gamma=T\})\defeq D.
    \]
    Therefore $\{b_\gamma\le T\}\subset B_{D}(\gamma(T))$. Clearly $\{b_\gamma\le T\}$ is closed, so it is compact.
\end{proof}

Next we will see something that typically happens for nonnegative sectional curvature as discussed in the introduction. It in particular shows that the asymptotic cone of a noncompact manifold with nonnegative Ricci curvature and linear volume growth is either a ray or a line. It is frequently attributed to \cite{SormaniSublinear} which proves the sublinear diameter growth of Busemann function level sets. However, we were not able to find a direct statement nor a proof of this fact. We include a proof here to make explicit how this fact follows from \cite{SormaniSublinear}. 

\begin{theorem}\label{cor:splitting}
    Let $(M,g)$ be an open $n$-manifold with $\Ric_g\ge 0$. Suppose $M$ has 
    linear volume growth \eqref{eq:linearvol}. Then the asymptotic cone of $M$ is unique and isometric to $\R$ or $[0,\infty)$. Moreover, $M$ splits if and only if the asymptotic cone is $\R$.
\end{theorem}

\begin{proof}
    If $M$ splits then obviously the asymptotic cone is unique and isometric to $\R$ with Euclidean metric. It suffices to show that if $M$ does not split then any asymptotic cone is isometric to a half line. Let $p\in M$, and suppose that for a sequence of radii $r_i\to \infty$, $(M,r_i^{-2}g,p)$ pGH converges to $(C,\dist, o)$. Take a ray $\gamma$ with $\gamma(0)=p$. By definition $\gamma$ converges under the rescaled metric $r_i^{-2}g$ to a ray $\gamma_\infty\subset C$, in the sense that $(M,r_i^{-2}g,\gamma(r_i t))\to (C,\dist, \gamma_\infty(t))$ for any $t\in[0,\infty)$. In particular $\gamma_\infty(0)=o$. We claim that $\gamma_\infty=C$. If not, there is a point $p_\infty\in C$ but not on $\gamma_\infty$. Let $a\defeq\dist(p_\infty,\gamma_\infty(0))>0$. Then there exists $\{p_i\in M\}_{i\in \N}$ so that $(M,r_i^{-2}g,p_i)\to (C,\dist, p_\infty)$ and $\dist_g(p_i,\gamma(0))\in [\frac i{i+1} ar_i,\frac {i+2}{i+1} ar_i]$. Let $t_i'\defeq b_\gamma(p_i)$. Since $M$ does not split, we see from Lemma \ref{lem:bddBuse} that $t_i'\to\infty$, so we can assume $t_i'\ge 0$. Furthermore, by Corollary \ref{cor:LevelClose} we have $|t_i'-0|\le\frac{i+2}{i+1}ar_i$, which in turn gives 
    \[
    \limsup_{i\to\infty} \frac {t_i'} {r_i}\in [0, a].
    \]
    It follows that
    \[
    \limsup_{i\to\infty} \frac{\dist_g(p_i,\gamma(t_i'))}{r_i}\le \limsup_{i\to\infty} \frac{\mr{diam}_g(b_\gamma^{-1}(t_i'))}{t_i'}\frac{t_i'}{r_i}=0.
    \]
    The last equality comes from \cite{SormaniSublinear}*{Theorem 1}. Notice that for some $t_\infty\in[0,a]$, up to subsequence $(M,r_i^{-2}g, \gamma(t_i'))\to (C,\dist,\gamma_\infty(t_\infty))$. Meanwhile $p_i\to p_\infty$, we see that $\dist(p_\infty,\gamma_\infty(t_\infty))=0$, a contradiction to our assumption that $p_\infty$ is not on $\gamma_\infty$.
\end{proof}

\subsection{The almost spliting theorem}\label{sec:AST}
We then recall some essentials of Cheeger--Colding's almost splitting theorem as it is the fundamental tool we rely on in the next section. The following discussion will be largely based on Cheeger's book \cite{CheegerRicBook}. We will only be concerned with the case of nonnegative Ricci curvature and the dimension $n\ge 2$. 

Let $(M^n,g)$ be an $n$-manifold with $\Ric_g\ge0$. For two given points $q^-,q^+\in M$ we define the excess of a point $p\in M$ with respect to $q^-$, $q^+$ as $e(p)\defeq \dist_g(p,q^-)+\dist_g(p,q^+)-\dist_g(q^-,q^+)$. We now describe the setting of the almost splitting at $p$. Let $L,\eps>0$ be such that 
\begin{align}
            \dist(p,q^+),\dist(p,q^-)&>L; \label{eq:distance}\\
                          e(p)&<\eps. \label{eq:excess}
\end{align}
Denote by $\Psi(\eps_1,\ldots,\eps_k|c_1,\ldots,c_m)$ a nonnegative error function so that when $c_1,\ldots, c_m$ are fixed, $\lim_{\eps_1,\ldots, \eps_k\to 0^+}\Psi=0$. Notice that $\Psi$ may change from line to line.

 Let $b^+(x)=\dist_g(p,q^+)-\dist_g(x,q^+)$. Fix a radius $R>0$. The Abresch-Gromoll inequality provides an upper bound for the excess function on $B_R(p)$.
 
 \begin{lemma}[\cite{AG},\cite{CheegerRicBook}*{Theorem 9.1}, \cite{XuLocalestimates}*{Lemma 2.2}]\label{lem:AbrGro}
    Let $(M^n,g)$ be an $n$-manifold with $\Ric_g\ge 0$. If \eqref{eq:distance}, \eqref{eq:excess} are satisfied then $\sup_{x\in B_R(p)}e(x)\le \Psi(\eps, L^{-1}|R,n)$. 
    
 \end{lemma}
 
 We define the harmonic replacement of $b^+$ on $B_R(p)$ by solving the Dirichlet problem
\begin{equation*}
    \begin{cases}
        \Delta\mathbf{b}^+&=0\ , \ \text{in $B_{4R}(p)$;}\\
        \mathbf{b}^+&=b^+,\ \text {on $\partial B_{4R}(p)$.}
    \end{cases}
\end{equation*}
Similar constructions can be done when the superscript $+$ is replaced by $-$, but we will not use them. 

\begin{lemma}[\cite{CheegerRicBook}*{Lemma 9.8},\cite{XuLocalestimates}*{(2.12)}]\label{lem:HarmApprox}
    Let $(M^n,g)$ be an $n$-manifold with $\Ric_g\ge 0$. If \eqref{eq:distance}, \eqref{eq:excess} are satisfied then $\sup_{B_{2R}(p)}|\mathbf{b}^+-b^+|\le \Psi(\eps,L^{-1}|R,n)$. 
    
\end{lemma}

The harmonic function $\mathbf{b}^+$ is called $(\Psi(\eps, L^{-1}|R,n),1)$-splitting function in the literature. It satisfies
\begin{enumerate}
    \item \cite{XuLocalestimates}*{(1.8)}. $\sup_{B_{2R}(p)}|\nabla \mathbf{b}^+|\le C(n)$ for some constant $C(n)>0$;
    \item \cite{CheegerRicBook}*{Lemma 9.10}. $\dashint_{B_{4R}(p)}|\nabla \mathbf{b}^+-1|^2\le \Psi(\eps, L^{-1}|R,n)$;
    \item \cite{CheegerRicBook}*{Lemma 9.13}. $R^2\dashint_{B_{2R}(p)}|\hess_{\mathbf{b}^+}|^2\le \Psi(\eps, L^{-1}|R,n)$.
\end{enumerate}
Under these conditions, it is noticed in several works that one can improve the gradient bound $\sup_{B_{2R}(p)}|\nabla \mathbf{b}^+|\le C(n)$ to
\begin{equation}\label{eq:SharpGradient}
    \sup_{B_{2R}(p)}|\nabla \mathbf{b}^+|\le 1+\Psi(\eps, L^{-1}|R,n).
\end{equation}
 See for example \cite{XuLocalestimates}*{Lemma 2.6} or \cite{HP23}*{Lemma 4.3} (substitute $L=\lambda=1$ therein).

The most important ingredient of the proof of the almost splitting theorem is the following almost Pythagorean theorem.

\begin{proposition}[\cite{CheegerRicBook}*{Lemma 9.16}]\label{prop:AlmostPyth}
    Let $x,z,w\in B_{R/2}(p)$, with $x\in (\mathbf{b}^+)^{-1}(a)$ for some $a\in \R$, and $z\in (\mathbf{b}^+)^{-1}(a)$ a choice of closest point to $w$, then
    \begin{equation}
        |\dist_g^2(x,z)+\dist_g^2(z,w)-\dist_g^2(x,w)|\le \Psi(\eps,L^{-1}|R,n).
    \end{equation}
\end{proposition}
It is worth mentioning that in its proof, we have the following relation between distances and levels of $\mathbf{b}^+$, as shown in the inequality between $(9.20)$ and $(9.21)$ of \cite{CheegerRicBook} as a consequence of Lemma \ref{lem:AbrGro} and Lemma \ref{lem:HarmApprox}, for $w,z$ in Proposition \ref{prop:AlmostPyth}, it holds  
\begin{equation}\label{eq:DistLevelClose}
    |\dist_g(w,z)-(\mathbf{b}^+(w)-\mathbf{b}^+(z))|\le \Psi(\eps,L^{-1}|R,n).
\end{equation}

We now state the almost splitting theorem. 
\begin{theorem}[\cite{CheegerRicBook}*{Theorem 9.25}]\label{thm:AlmostSplitting}
    Let $(M^n,g)$ be a $n$-manifold with $\Ric_g\ge 0$. Let $R>0$, $L>4R+1$, $\eps\in (0,1)$ so that $\eqref{eq:distance}$, $\eqref{eq:excess}$ are satisfied. Set 
    \begin{align*}
         F: B_R(p)&\to \R\times  (\mathbf{b}^+)^{-1}(0)\\
              w &\mapsto (\mathbf{b}^+(w),z),
    \end{align*}
  where $z\in (\mathbf{b}^+)^{-1}(0)$ is a closest point to $w$. Equip $(\mathbf{b}^+)^{-1}(0)$ with the ambient metric $\dist_g$ and $\R$ with the standard Euclidean metric. Take the product metric on $\R\times  (\mathbf{b}^+)^{-1}(0)$. Then $F$ is a $\Psi(\eps,L^{-1}|R,n)$-Gromov-Hausdorff approximation between $B_R(p)$ and $B_R((0,x))\subset \R\times (\mathbf{b}^+)^{-1}(0)$ where $x$ is a choice of closest point to $p$ and $\R\times (\mathbf{b}^+)^{-1}(0)$.
\end{theorem}

\begin{remark}\label{rmk:LevelsetConverge}
    We will also use the sequential version of the almost splitting theorem. Let $p_i, q_i^-, q_i^+$ be sequences of points so that the excess of $p_i$ with respect to $ q_i^-, q_i^+$ tends to $0$, and $L_i\defeq\min\{\dist_g(p_i,q_i^-),\dist_g(p_i,q_i^+)\}\to \infty$. Then $(M,g, p_i)$ pGH converges to $(\R\times X, \dist, (0,x))$. In particular, let $b_i^+=\dist_g(p_i,q_i^+)-\dist_g(\cdot,q_i^+)$ there exists some $R_i\to\infty$ with $L_i>4R_i+1$ so that if $\mathbf{b}_i^+$ is the harmonic replacement on $B_{R_i}(p_i)$ of $b_i^+$ and $x_i$ is a choice of closest point of $p_i$ to $(\mathbf{b}_i^+)^{-1}(0)$, then $(\{\mathbf{b}_i^+=0\},\dist_g, x_i)$ pGH converges to $(X,\dist, x)$. Here the distance $\dist$ on the product space $\R\times X$ restricts to $X$. 
\end{remark}

\begin{remark}
    If $(M,g)$ with $\Ric_g\ge0$ is also noncollapsed in the sense of \eqref{eq:ncends}, then the $X$ we get from Remark \ref{rmk:LevelsetConverge} is a noncollapsed $\RCD(0,n-1)$ ($\ncRCD(0,n-1)$ in short) space. To show this we need to incorporate the splitting theorem of $\RCD$ spaces \cite{Gigli_splitting} and the volume convergence theorem of De Philippis--Gigli \cite{DPG17}. We will use this observation without further mentioning the argument.
\end{remark}

To close this section, we derive two consequences of the almost splitting theorem along a given ray. The first is a quantitative way of saying that the $\R$-factor split off in the limit space at infinity comes from the ray along which the limit is take. The second is to bound the diameters of Busemann function level sets given the information at infinity.  

\begin{proposition}\label{prop:RayConverge}
    Let $(M^n,g)$ be an open $n$-manifold with $\Ric_g\ge 0$ and $\gamma:[0,\infty)\to M$ be a ray. If for a divergent sequence $t_i\to \infty$, $(M,g,\gamma(t_i))$ pGH converges to $(\R\times X, \dist, (0,x))$ then for any $a\in \R$, $\gamma(t_i+a)$ converges to $(a,x)$ along with this pGH convergence.
\end{proposition}

\begin{proof}
We can assume $a>0$ otherwise we take $|a|$. Let $p_i=\gamma(t_i)$, $p'_i=\gamma(t_i+a)$, $q_i^-=\gamma(0)$, and $q_i^+=\gamma(T_i)$ for some $T_i>t_i+16a+8\to\infty$. Then we define $b_i^+(p)=(T_i-t_i)-\dist_g(p, q_i^+)$, and $\mathbf{b}_i^+$ to be the harmonic replacement of $b_i^+$ in $B_{16a+8}(p_i)$. Denote by $x_i\in (\mathbf{b}_i^+)^{-1}(0)$ a choice of closest point to $p_i$ and $z_i\in (\mathbf{b}_i^+)^{-1}(0)$ a choice of closest point to $p'_i$. From the construction in Theorem \ref{thm:AlmostSplitting}, it suffices to prove that $\mathbf{b}_i^+(p'_i)\to a$ and $z_i\to x$ along $(M,g,\gamma(t_i))\to (\R\times X, \dist, (0,x))$. 

it is clear that the excess of $\gamma(t_i)$ with respect to $q_i^-$, $q_i^+$ is $0$. In what follows we will use the error function $\Psi$ which is now independent of the excess and we abbreviate $\Psi(t_i^{-1},(T_i-t_i)^{-1}|a,n)$ as $\Psi_i$. We first estimate using \eqref{eq:DistLevelClose} and Lemma \ref{lem:HarmApprox} that
\[
\dist_g(p_i,x_i)\le |\mathbf{b}_i^+(p_i)-\mathbf{b}_i^+(x_i)| +\Psi_i=|\mathbf{b}_i^+(p_i)|+\Psi_i\le |b_i^+(p_i)|+2\Psi_i=2\Psi_i.
\]
Then,since $\dist_g(p'_i,p_i)=\dist_g(\gamma(t_i+a),\gamma(t_i))=a$, by the gradient estimate \eqref{eq:SharpGradient} we get 
\begin{equation}\label{eq:DistEst1}
    |\mathbf{b}_i^+(p'_i)-\mathbf{b}_i^+(p_i)|\le \int_{t_i}^{t_i+a}|\nabla\mathbf{b}_i^+(\gamma(s))||\gamma'(s)|d s \le (1+\Psi_i)a=a+\Psi_i,
\end{equation}
 and from Lemma \ref{lem:HarmApprox} we infer that 
 \begin{equation}\label{eq:DistEst2}
     |\mathbf{b}_i^+(p_i)|=|\mathbf{b}_i^+(p_i)-b_i^+(p_i)|\le \Psi_i.
 \end{equation}
  Combining together \eqref{eq:DistEst1} and \eqref{eq:DistEst2}, we have $|\mathbf{b}_i^+(p'_i)-a|\le 2 \Psi_i$. It in turn gives 
 \begin{equation}\label{eq:DistEst3}
     |\dist_g(p'_i,z_i)-a|\le 3\Psi_i,
 \end{equation}
   by \eqref{eq:DistLevelClose}. When $i$ is large we will have $a-3\Psi_i>0$ so $\dist_g^2(p'_i,z_i)\ge (a-3\Psi_i)^2$. 

We also observe that $z_i\in B_{2a+1}(p_i)$ for large $i$. Indeed, incorporating \eqref{eq:DistEst3} we have 
\begin{align*}
    \dist_g(p_i,z_i)\le \dist_g(p_i, p'_i)+\dist_g(p'_i,z_i)\le a+a+3\Psi_i\le 2a+1,
\end{align*}
 when $i$ is large. Now by the almost Pythagorean theorem, Proposition \ref{prop:AlmostPyth}, we have 
\begin{align*}
    \dist_g^2(x_i,z_i)&\le \dist_g^2(p'_i,x_i)-\dist_g^2(p'_i,z_i)+\Psi_i^2\le (\dist_g(p_i,x_i)+\dist_g(p'_i,p_i))^2-(a-3\Psi_i)^2+\Psi_i^2\\
                      &\le (a+2\Psi_i)^2-(a-3\Psi_i)^2+\Psi_i^2\le \Psi_i.
\end{align*}
This completes the proof that $\gamma(t_i+a)\to (a,x)$ since by Theorem \ref{thm:AlmostSplitting}, Remark \ref{rmk:LevelsetConverge} $x_i\to x$ along $(M,g,\gamma(t_i))\to (\R\times X, \dist, (0,x))$. 



\end{proof}

\begin{proposition}\label{prop:LevelDiam}
       Let $(M^n,g)$ be an open $n$-manifold with $\Ric_g\ge 0$ and $\gamma:[0,\infty)\to M$ be a ray. If for any divergent sequence $t_i\to \infty$, up to subsequence $(M,g,\gamma(t_i))$ pGH converges to $(\R\times X, \dist, (0,x))$ for some compact $X$, then $\limsup_{t\to\infty} \mr{diam}_g(b_\gamma^{-1}(t))< \infty$.
\end{proposition}

\begin{proof}
Let $A=\{X:\text{ $\R\times X$ is a pGH limit of $(M,g,\gamma(t_i))$ for some $t_i\to \infty$}\}$. We observe that $\sup_{X\in A}\mr{diam}(X)<\infty$. Indeed, if not we have a contradicting sequence $X_j$ so that $\mr{diam}(X_j)\to \infty$ and each $X_j$ is a pGH limit of $(M,g,\gamma(t_{j,i}))$ for $t_{j,i}\to \infty$ as $i\to\infty$. Then we can find a diagonal sequence denoted by $t'_{i,i}\to\infty$ so that up to subsequence $(M,g, \gamma(t'_{i,i}))$ converges to $\R\times X$, but this $X$ can not be compact since it is an $\RCD(0,n-1)$ with infinite diameter, a contradiction. Set $D\defeq \sup_{X\in A}\mr{diam}(X)$.

Now we bound the diameters of level sets of $b_\gamma$. We first set up the almost splitting along the ray $\gamma$. For some fixed $t>0$, let $T>t$ to be chosen, let $p=\gamma(t)$, $q^-=\gamma(0)$, $q^+=\gamma(T)$, and $b_t^+=\dist_g(p, q^+)-\dist_g(\cdot,q^+)$. Take $R=10D+10$ and let $\mathbf{b}_t^+$ be the harmonic replacement of $b_t^+$ in $B_R(p)$. Take $y_{t,1},y_{t,2}\in b_\gamma^{-1}(t)\cap B_{R/2}(p)$. Since here we have $e(p)=0$, $\dist_g(p,q^+)=T-t$, we can use the error function $\Psi(t^{-1},(T-t)^{-1}|R, n)$ abbreviated as $\Psi$. For each $t>0$, there exists $T(t)$ depending on $t$ so that 
$$|b_t^+(y_{t,j})|=|b_t^+(y_{t,j})-(b_\gamma(y_{j})-t)|\le \Psi, \ j=1,2.$$ 
This is possible thanks to the local uniform convergence of $$(T-t)-\dist_g(y_{j},\gamma(T))\to b_\gamma(y_{j})-t$$ as $T\to\infty$ for each fixed $t>0$, $j=1,2$. 
From Remark \ref{rmk:LevelsetConverge} we see that $$\limsup_{t\to\infty}\mr{diam}_g((\mathbf{b}_t^+)^{-1}(0))\le D.$$ By Lemma \ref{lem:HarmApprox}, we have that
$$|b_t^+(y_{t,j})-\mathbf{b}_t^+(y_{t,j})|\le \Psi,\ j=1,2.$$ 
Altogether we have $|\mathbf{b}_t^+(y_{t,j})|\le 2\Psi$. Let $y_{t,j}'$ be a choice of closest point on $(\mathbf{b}_t^+)^{-1}(0)$ to $y_{t,j}$, $j=1,2$. we can estimate that
\begin{align}\label{eq:UniBoundLevel}
    \dist_g(y_{t,1},y_{t,2})&\le \dist_g(y_{t,1},y_{t,1}')+\dist_g(y_{t,1}',y_{t,2}')+\dist_g(y_{t,2},y_{t,2}')\\
    &\le |\mathbf{b}_t^+(y_{t,1})|+\mr{diam}_\dist((\mathbf{b}_t^+)^{-1}(0))+|\mathbf{b}_t^+(y_{t,2})|+2\Psi\notag\\
    &\le \mr{diam}_\dist((\mathbf{b}_t^+)^{-1}(0))+6\Psi.\notag
\end{align}
Then since $y_{t,1},y_{t,2}$ are arbitrary, we get from \eqref{eq:UniBoundLevel} that 
\[
\mr{diam}_g(b_\gamma^{-1}(t)\cap B_{R/2}(p))\le \mr{diam}_\dist((\mathbf{b}^+)^{-1}(0))+6\Psi(t^{-1}, (T(t)-t)^{-1}|R,n).
\]
Taking $t\to\infty$, hence $T(t)\to\infty$, we obtain that
\begin{equation}\label{eq:LevelDiam0}
    \limsup_{t\to\infty} \mr{diam}_g(b_\gamma^{-1}(t)\cap B_{R/2}(\gamma(t)))\le D.
\end{equation}
There exists $t_0>0$ such that for $t\ge t_0$, $\mr{diam}_g(b_\gamma^{-1}(t)\cap B_{R/2}(p))\le D+1$. We claim that $b_\gamma^{-1}(t)\cap B_{R/2}(\gamma(t))=b_\gamma^{-1}(t)$ holds for all $t\ge t_0+1$. If not, suppose there exists $t\ge t_0+1$ and $y\in b_\gamma^{-1}(t)\setminus B_{R/2}(\gamma(t))$. We take $\tau_1\ge t$ so that $\tau_1-\dist_g(y,\gamma(\tau_0))\ge b_\gamma(y)-1$. Let $\sigma(s):[0,1]\to M$ be a geodesic with $\sigma(0)=y$ and $\sigma(1)=\gamma(\tau_0)$. Then for all $s\in[0,1]$, 
\begin{equation}\label{eq:LevelDiam1}
    b_\gamma(\sigma(s))\ge \tau_1-\dist_g(\sigma(s),\gamma(\tau_1))\ge \tau_1-\dist_g(\sigma(0),\gamma(\tau_1))\ge b_\gamma(y)-1\ge t_0.
\end{equation}
Here we have used that $s\mapsto s-\dist_g(\cdot, s)$ is monotone increasing. To ease the notation we set $\tau_s=b_\gamma(\sigma(s))$ for $s\in[0,1]$. In particular $\tau_0=t$. By the assumption, it holds that 
\begin{equation}\label{eq:LevelDiam2}
    \mr{diam}_g(b_\gamma^{-1}(\tau_s)\cap B_{R/2}(\gamma(\tau_s)))\le D+1,\ \forall s\in[0,1].
\end{equation}
On the other hand we can estimate $\dist_g(\sigma(s), \gamma)$ for $s\in[0,1]$. Let $y'$ be a choice of closest point of $y$ on $\gamma$ then $5D+5=R/2\le \dist_g(y,\gamma(t))\le \dist_g(y,y')+\dist_g(y',\gamma(t))$. By definition $\dist_g(y,y')=\dist_g(y,\gamma)$ and by Corollary \ref{cor:LevelClose}, $\dist_g(y',\gamma(t))\le \dist(y,y')=\dist_g(y,\gamma)$, so $\dist(y,\gamma)\ge 2.5D+2.5$. Observe also that $\dist_g(\cdot, \gamma)$ is a continuous function and $\dist_g(\sigma(1),\gamma)=0$, by intermediate value theorem there is a $s_0\in (0,1)$ so that $\dist_g(\sigma(s_0),\gamma)=2D+2$. Then by Corollary \ref{cor:LevelClose} we have 
$$
\dist_g(\sigma(s_0), \gamma(\tau_{s_0}))\le 4D+4<R/2.
$$ 
It immediately follows from \eqref{eq:LevelDiam1} and \eqref{eq:LevelDiam2} that $\dist_g(\sigma(s_0),\gamma(\tau_{s_0})\le D+1$. However, this is a contradiction to $\dist_g(\sigma(s_0),\gamma)=2D+2$ we have derived. Thus we have shown that $b_\gamma^{-1}(t)\cap B_{R/2}(\gamma(t))=b_\gamma^{-1}(t)$ holds for all $t\ge t_0+1$. Combine this with \eqref{eq:LevelDiam0} we have completed the proof.

\end{proof}

\section{Structure of limit spaces at infinity}\label{sec:infinity}

In this section the final goal is to prove Theorem \ref{thm:equi}. We need to further derive some properties of Busemann functions under the volume noncollapsed assumption.

The first lemma is standard, see \cite{antonelli2023isoperimetric}*{Corollary 3.11} or \cite{SchoenYauBook}. We provide a proof here for completeness. In what follows, for any $R>r>0$ and $p\in M$, we denote the annulus $B_R(p)\setminus B_r(p)$ by $\mr{Ann}_{R,r}(p)$. 

\begin{lemma}\label{lem:AnnulusVol}
Let $(M,g)$ be an open $n$-manifold with $\Ric_g\ge 0$. Then for every $p\in M$ and $R>r>0$, it holds
\begin{equation}
    \volg(\mr{Ann}_{R,r}(p))\le n \frac{\volg(B_R(p))}{R}\frac Rr (R-r).
\end{equation}
Furthermore, if $(M,g)$ has linear volume growth \eqref{eq:linearvol}, and is noncollapsed \ref{eq:ncends} then for any sequence $\{p_i\}_{i\in \N^+}$ diverging to infinity, any pGH limit space $(X,\dist,\haus^n)$ of the sequence $(M,g,p_i)$ has linear volume growth with 
\begin{equation}\label{eq:volesti}
    \haus^n(B_\rho(x))\le 2nV\rho.
\end{equation}
In particular if $X$ splits isometrically as $(\R\times K, \haus^1\times \haus^{n-1})$, then $\haus^{n-1}(K)\le nV$ and $K$ is compact. 
\end{lemma}

    It is not always the case that the same volume growth order is carried to a limit space at infinity. For example the paraboloid $\{(x,y,z)\in \R^3: z=x^2+y^2\}$ has volume growth of order $\frac32$, while its limit space at infinity is $\R^2$ regardless of divergent sequence which has volume growth order $2$.

\begin{proof}
    Let $\varphi(x): M\to \R$ be the following bounded Lipschitz function with compact support.
    \begin{equation}
        \varphi(x)=\begin{cases}
1 &\dist_{g,p}(x)\le r,\\
\frac1{R-r}(R-\dist_{g,p}(x)) & r<\dist_{g,p}(x)<R,\\
0 &\dist_{g,p}(x)\ge R.
        \end{cases}
    \end{equation}
    The distributional Laplacian comparison $\mathbf{\Delta}\dist_p^2\le 2n\ \volg$ implies that 
    \begin{align}
        2n\volg(B_R(p))&\ge \int_{B_R(p)}\varphi \di \mathbf{\Delta}\dist_{g,p}^2=-\int_{B_R(p)} \nabla \varphi \cdot 2\dist_{g,p}\nabla \dist_{g,p}\di \volg\notag\\
        &=2\int_{\mr{Ann}_{R,r}(p)}\frac{\dist_{g,p}|\nabla \dist_{g,p}|^2}{R-r}\di \volg\ge \frac{2r}{R-r}\volg(B_R(p)\setminus B_r(p)).
    \end{align}
     So we can rearrange terms in the above inequality to derive that
    \begin{equation}\label{eq:annulus}
        \volg(\mr{Ann}_{R,r}(p))\le n \frac{\volg(B_R(p))}{R} \frac Rr (R-r)
    \end{equation}

    Now given $\{p_i\}_{i\in N^+}$ diverging to infinity. We may assume that $(M,p_i)$ converges to $(X,x)$ up to taking a subsequence. Fix a radius $\rho\in \R^+$ then for large enough $i$, $\rho\in (0,\dist(p, p_i))$. Applying \eqref{eq:annulus} to $R=\dist(p, p_i)+\rho$, $r=\dist(p,p_i)-\rho$ yields that 
    \begin{equation}
        \volg(B_\rho(p_i))\le \volg(\mr{Ann}_{R,r}(p))\le n\frac{\volg(B_R(p))}{R}\frac{\dist(p, p_i)+\rho}{\dist(p,p_i)-\rho}\cdot 2\rho.
    \end{equation}
    Let $i\to \infty$, we get from the volume convergence \cites{Colding97,Cheeger-Colding97I, DPG17} and volume growth assumption \eqref{eq:linearvol} that 
    \begin{equation}
        \haus^n(B_\rho(x))\le 2nV\rho.
    \end{equation}
   This is the sought estimates. If $X=\R\times K$ isometrically, then we have that $\haus^{n-1}(K)=\lim_{\rho\to \infty}\frac{\haus^n(B_\rho(x))}{2\rho}\le nV$. Since $K$ is $\ncRCD(0,n-1)$ with finite $\haus^{n-1}$-volume, it is compact.
\end{proof}

We will turn to the study of Busemann functions with the noncollapsed condition \eqref{eq:ncends}. Some pathological behaviors are excluded as a result of this extra assumption. Note that the proposition we will prove below is not possible without \eqref{eq:ncends} as mentioned in Remark \ref{rmk:ncends}.    

\begin{proposition}\label{prop:bddDist}
    Let $(M,g)$ be an open $n$-manifold with $\Ric_g\ge 0$. Suppose $M$ is noncollapsed \eqref{eq:ncends} and has linear volume growth \eqref{eq:linearvol}, then either $M$ splits or for any ray $\gamma$, there exists a constant $D>0$ such that 
    \[
    \dist_g(p,\gamma)\le D, \ \forall p\in M.
    \]
    
\end{proposition}

\begin{proof}
    Assume the statement is not true, then there exists a sequence of points $\{p_i\}_{i\in\N^+}$ such that $\dist_g(p_i,\gamma)\to \infty$ as $i\to \infty$. Clearly $\{p_i\}_{i\in\N^+}$ diverges to infinity. There exists $t_i\in [0,\infty)$ such that $\dist_g(p_i,\gamma)=\dist_g(p_i,\gamma(t_i))$. Let $\sigma_i$ be a geodesic joining $p_i$ with $\gamma(t_i)$. There are now two cases.
    \begin{enumerate}
        \item\label{item:BddDist 1} If up to subsequence $t_i\to \infty$. The limit space $(X,\dist, x)\defeq\lim_{i\to \infty}(M,g,\gamma(t_i))$ splits off a line coming from $\gamma$ by the almost splitting theorem, Theorem \ref{thm:AlmostSplitting} and Proposition \ref{prop:RayConverge}. Write $X=\R\times K$, where $K$ is an $\ncRCD(0,n-1)$ space. $K$ is compact thanks to Lemma \ref{lem:AnnulusVol}. We set $x=(0,k_0)$. The segments $\sigma_i$ locally uniformly converge to a ray $\sigma$ emanating from $x$, and for any $s\in[0,\infty)$, $\dist(\sigma(s),x)=s$. 
        
        We claim that for any $s\in[0,\infty)$, $\dist(\sigma(s),\R\times \{k_0\})=s(=\dist(\sigma(s),x))$. Suppose not, there is a $s_0>0$ and a $\eps_0>0$ such that $\dist(\sigma(s_0),\R\times \{k_0\})<s_0-\eps_0$. Set $\sigma(s)=(\sigma_\R(s),\sigma_K(s))$, then $s_0-\eps_0>\dist(\sigma(s_0),\R\times \{k_0\})=\dist(\sigma(s_0),(\sigma_\R(s_0),k_0))$. The last equality comes from the product metric structure. By proposition \ref{prop:RayConverge}, $\gamma(t_i+\sigma_\R(s_0))\to (\sigma_\R(s_0),k_0))$. Meanwhile, by the choice of $t_i$, 
        $$
        s_0=\dist_g(\sigma_i(s_0),\gamma(t_i))\le \dist_g(\sigma_i(s_0),\gamma(t_i+\sigma_\R(s_0)))\le s_0-\frac{\eps_0}2,
        $$ 
        when $i$ is large, a contradiction. This completes the proof of the claim. 
        
        However, our claim implies $$\dist^2((\sigma_\R(s),\sigma_K(s)),(\sigma_\R(s),k_0))=\dist^2((\sigma_\R(s),\sigma_K(s)),(0,k_0)),$$ which simplifies to $\sigma_\R(s)=0$ for any $s\in[0,\infty)$. This implies that $\sigma\subset \{0\}\times K$, which is impossible since $K$ is compact.
        
        \item\label{item:BddDist 2} If up to subsequence $t_i\to t\in[0,\infty)$, then up to subsequence $\sigma_i$ converges to a ray $\sigma$ emanating from $\gamma(t)$. After a reparameterization, we can take $t=0$. The convergence of $\sigma_i$ yields that 
        \begin{equation}\label{eq:perp}
            \dist_g(\sigma(s),\gamma)=\dist_g(\sigma(s),\gamma(0))=s, \ \forall s\ge 0.
        \end{equation}
           We show in this case $M$ must split. If $M$ does not split, then $\{b_\sigma\le 0\}$ is compact due to Lemma \ref{lem:bddBuse}. However, for any $T\ge 0$, we see from \eqref{eq:perp} that
         \[
         s-\dist(\gamma(T),\sigma(s))=\dist(\gamma(0),\sigma(s))-\dist(\gamma(T),\sigma(s)) \le 0,\ \forall s\ge 0.
         \]
         Let $s\to \infty$, it follows that $\gamma\subset \{b_\sigma\le 0\}$ which contradicts the compactness of $\{b_\sigma\le 0\}$.
    \end{enumerate}
\end{proof}

It is expected that the constant $D$ is independent of the ray $\gamma$. 

Motivated by Proposition $\ref{prop:bddDist}$, we also consider that for a fixed ray $\gamma$ there exists $D>0$ such that for any $p\in M$, $\dist_g(p,\gamma)<D$ as an assumption. This assumption is more general, for example it is allowed that the end of $M$ collapses to $\R$. It serves as a weaker substitution of linear volume growth \eqref{eq:linearvol} and noncollapsed condition \eqref{eq:ncends}. We will refer to it as the bounded distance to a fixed ray condition. 

\begin{lemma}\label{lem:NoSplitting}
     Let $(M^n,g)$ be an open $n$-manifold with $\Ric_g\ge0$, and $\gamma$ be a ray in $M$. If there exists $D>0$ such that for any $p\in M$, $\dist_g(p,\gamma)<D$, then $M$ does not split and the Busemann function $b_\gamma$ associated to $\gamma$ has finite minimum and its level sets have uniform diameter bound $4D$.
\end{lemma}

\begin{proof}
        If $(M,g)$ splits, write $(M,g)=(\R\times K,g)$. Also we can set $\gamma(0)=(0,k_0)$. Take $a,b\in[0,\infty)$, let $t_a\ge 0$ (resp. $t_b\ge 0$) be such that $\dist_g((a,k_0),\gamma)=\dist_g((a,k_0),\gamma(t_a))$ (resp. $\dist_g((-b,k_0),\gamma)=\dist_g((-b,k_0),\gamma(t_b))$) and $t_a'\ge0$ (resp $t_b'\ge0$) be such that $b_\gamma((a,k_0))=t_a'$ (resp. $b_\gamma((-b,k_0))=t_b'$). Using our assumption and applying Corollary \ref{cor:LevelClose}, we can estimate that 
        
    \begin{align}\label{eq:PosNeg}
        a+b&=\dist_g((a,k_0),(-b, k_0))\\
           &\le \dist_g((a,k_0),\gamma(t_a))+|t_a-t_a'|+|t_a'-t_b'|+\dist_g((-b,k_0),\gamma(t_b))+|t_b-t_b'|\notag\\
           &\le 4D+|t_a'-t_b'|.\notag
    \end{align}
    
     Now fix $b$ and let $a\to \infty$ in \eqref{eq:PosNeg}, we see that $b_\gamma((a,k_0))=t_a'\to \infty$. Likewise, fix $a$ and let $b\to\infty$ in \eqref{eq:PosNeg} we see that $b_\gamma((-b,k_0))=t_b'\to\infty$. Then since $t\mapsto b_\gamma((t,k_0))$ is continuous. We find by intermediate value theorem that there exists $t_i,s_i\to\infty$ such that 
     $$
     b_\gamma((t_i,k_0))=b_\gamma((-s_i,k_0))\to\infty.
     $$ In particular $\limsup_{t\to\infty}\mr{diam}_g(b_\gamma^{-1}(t))=\infty$. However, we will prove that the diameters of the level sets of $b_\gamma$ are uniformly bounded. This will be a contradiction. 
    
    To this end, we first show that $b_\gamma$ has a finite minimum. Given $p\in M$ let $t_p\ge0$ be such that $\dist(p,\gamma)=\dist(p,\gamma(t_p))$, then by Corollary \ref{cor:LevelClose}, $|b_\gamma(p)-t_p|\le D$, which in turn gives $b_\gamma(p)\ge -D+t_p\ge -D$, as desired. Let $m=\min_{x\in M}b_\gamma(x)$, and $p,q\in b_\gamma^{-1}(T)$ for $T\in[m,\infty)$. Let $t_p$ (resp. $t_q$) be such that $\dist(p,\gamma)=\dist(p,\gamma(t_p))$ (resp. $\dist(q,\gamma)=\dist(q,\gamma(t_q))$), Corollary \ref{cor:LevelClose} implies $|t_p-t_q|\le |t_p-b_\gamma(p)|+|t_q-b_\gamma(q)|\le 2D$. It follows from our assumption that 
    \[
\dist_g(p,q)\le \dist(p,\gamma(t_p))+\dist(q,\gamma(t_q))+|t_p-t_q|\le 4D. 
    \]
    This estimate is independent of $p,q,T$, we take take supremum over $p,q\in b_\gamma^{-1}(T)$ to get $\mr{diam}_g(b_\gamma^{-1}(T))\le 4D$ and then take supremum over $T\ge m$. This completes the proof.
\end{proof}


\begin{lemma}\label{lem:limitBdd}
    Let $(M^n,g)$ be an open $n$-manifold with $\Ric_g\ge0$, and $\gamma$ be a ray in $M$. If there exists $D>0$ such that for any $p\in M$, $\dist_g(p,\gamma)<D$, then for any $t_i\to \infty$ such that $(M,g,\gamma(t_i))\to (\R\times K, \dist, (0,k_0))$, $K$ is compact. 
\end{lemma}

\begin{proof}
     It suffices to show $\dist(k_0,k)\le D$ for every $k\in K$. Here the distance $\dist$ on $\R\times K$ restricts to $\{0\}\times K$, which is isometric to $K$. Fix $k\in K$ then there is a diverging sequence $p_i\in M$ such that $(M,g,p_i)\to (\R\times K,\dist, (0,k))$. Let $s_i$ be such that $\dist_g(p_i,\gamma)=\dist_g(p_i,\gamma(s_i))<D$. 
     We have
     \[
     |t_i-s_i|=\dist_g(\gamma(t_i),\gamma(s_i))\le \dist_g(\gamma(t_i),p_i)+\dist_g(p_i,\gamma(s_i))\le \dist_g(\gamma(t_i),p_i)+D.
     \]
     Note that $\dist_g(p_i,\gamma(t_i))\to \dist(k,k_0)$ so $\dist_g(p_i, \gamma(t_i))$ are uniformly bounded, in particular, $s_i\to \infty$ as $t_i\to\infty$. We also infer that there exists a subsequence which we do not relabel so that $\lim_{i\to\infty} s_i-t_i\defeq a$ exists and $a\in[-\dist(k,k_0)-D,\dist(k,k_0)+D]$. By Proposition \ref{prop:RayConverge}, $(M,g,\gamma(t_i+a))\to (\R\times K, \dist, (a,k_0))$. Also notice that up to subsequence $\dist_g(\gamma(s_i),\gamma(t_i+a))\to 0$, thus we have $(M,g,\gamma(s_i))\to (\R\times K, \dist, (a,k_0))$. Letting $i\to\infty$ in $\dist_g(p_i,\gamma(s_i))<D$, we see that $\dist((0,k),\dist(a,k_0))\le D$. This reads $a^2+\dist^2(k,k_0)\le D^2$, yielding the desired bound $\dist(k,k_0)\le D$. 
\end{proof}

The next proposition provides a new perspective to understand Sormani's sublinear diameter growth of the level sets of a Busemann function \cite{SormaniSublinear}. It confirms the intuition in Remark \ref{rmk:ncends}. 
\begin{proposition}\label{prop:uniBddDiam}
     Let $(M,g)$ be an open $n$-manifold with $\Ric_g\ge 0$ and $\gamma$ be a ray in $M$ and $b_\gamma$ be the associated Busemann function. If $(M,g)$ is noncollapsed \eqref{eq:ncends} and has linear volume growth \eqref{eq:linearvol}, then $\mr{diam}(b_\gamma^{-1}(t))$ is uniformly bounded for any $t\in \R$.
\end{proposition}

\begin{proof}
    If $M$ splits as $\R\times K$, then $K$ is a compact $(n-1)$-manifold of nonnegative Ricci curvature. Notice that the translations in the $\R$-factor are isometries, so we can bring any point into $\{0\}\times K$ via an isometry, so the limit space at infinity is unique and it is $M=\R\times K$. The uniform bound on $\mr{diam}(b_\gamma^{-1}(t))$ now follows from Proposition \ref{prop:LevelDiam} and Lemma \ref{lem:Sormani}. 
    
    If $M$ does not split, then by proposition \ref{prop:bddDist} there is $D>0$ such that for any $p\in M$, $\dist_g(p,\gamma)<D$. The desired result follows from the proof of Lemma \ref{lem:NoSplitting}.

\end{proof}

\begin{remark}
    In fact, in the previous Proposition, if $M$ splits as $\R\times K$ for $K$ compact, then any ray lies entirely in $\R\times\{k\}$ for some $k\in K$. This is because the projection of any ray onto $K$ is a minimizing geodesic parametrized on $[0,\infty)$, hence is either a ray or a point, and it cannot be a ray since $K$ is compact. In this case all the Busemann function level sets are isometric to $K$ hence trivial have uniformly bounded diameter. 
\end{remark}

Now we can provide a proof of the existence of splitting for any sequence of points diverging to infinity in the setting of Theorem \ref{thm:splitting}. We will postpone the proof of equal volume of the cross sections to Proposition \ref{prop:conv}.

\begin{proof}[First part of proof of Theorem \ref{thm:splitting}]
    Suppose that $M$ does not split, otherwise let $M=\R\times K$ for some compact $K$. As already observed, by translations along the $\R$-factor, we see that the unique limit space at infinity of $M$ is $\R\times K$ regardless of divergent sequences, and $$\haus^{n-1}(K)=\lim_{r\to\infty}\frac{\volg(B_r(p))}{2r}= \frac V2.$$  
    
    Take a ray $\gamma$, and the associated Busemann function $b_\gamma$. Let $\{p_i\}$ be an arbitrary sequence, Proposition \ref{prop:bddDist} asserts that there exists $D>0$ such that $\dist_g(p_i,\gamma)<D$. Let $t_i= b_\gamma(p_i)$. It is readily checked that 
    $$
        t_i=\dist_g(\gamma(0),\gamma(t_i))\ge \dist_g(\gamma(0),p_i)-\dist_g(\gamma(t_i),p_i)\ge \dist_g(\gamma(0),p_i)- \mr{diam}(b_\gamma^{-1}(t_i))\to \infty,
    $$ 
    given the uniform boundedness of the level sets of $b_\gamma$ due to Proposition \ref{prop:uniBddDiam}. Meanwhile, 
    For large $i$ so that $\dist_g(p_i,\gamma(0))>D$, choose large $T_i$ so that $\dist_g(p_i,\gamma(T_i))>D$ and that $T_i\to\infty$. The Abresch-Gromoll excess estimate \cite{AG}*{Proposition 2.3} asserts that the excess of $p_i$ with respect to $\gamma|_{[0,T_i]}$ is controlled as follows 
    \[
      e(p_i)\le C(n)D^{\frac n{n-1}}\left( \frac1{\dist_g(p_i,\gamma(0))-D}+\frac1{\dist_g(p_i,\gamma(T_i))-D} \right)^{\frac 1{n-1}}.
    \]
    In particular, $e(p_i)\to 0$ as $i\to\infty$, then by Remark \ref{rmk:LevelsetConverge} up to subsequence the limit space $(X,\dist,x)$ of $(M,g, p_i)$ splits off a line. We write $X=\R\times K$. The compactness of $K$ and $\haus^{n-1}(K)$ upper bound follows from Lemma \ref{lem:AnnulusVol}.
\end{proof}


It then comes to the core of this note. Lemma \ref{lem:uniclose} below allows us to study the isoperimetric problem with the exact same strategy as in \cite{antonelli2023isoperimetric}. It is the most difficult step to generalize the results in \cite{antonelli2023isoperimetric} from $\sec\ge0$ to $\Ric\ge 0$. 

\begin{lemma}\label{lem:uniclose}
     Let $(M,g)$ be an open $n$-manifold with $\Ric_g\ge 0$. Let $\gamma$ be a ray in $M$ and $b_\gamma$ be the associated Busemann function. Suppose there exists $D>0$ such that $d\defeq\dist_g(p,\gamma)<D$ for every $p\in M$, 
     and $M$ does not split, then for any sufficiently small $\eps>0$, there exists $R>0$ such that for $p\in M$ satisfying $\dist_g(p,\gamma(0))>R$ it holds $\dist_g(\gamma(b_\gamma(p)),\gamma(t_p))=|b_\gamma(p)-t_p|<\eps$, where $t_p$ is such that $\dist_g(p,\gamma)=\dist_g(p,\gamma(t_p))$.
\end{lemma}

\begin{proof}
    We set $t_p'\defeq b_\gamma(p)$. 
      We first observe from Corollary \ref{cor:LevelClose} that 
      \begin{equation}\label{eq:ttprimeClose}
          |t_p-t_p'|\le D. 
      \end{equation}
       Then we prove by contradiction. Assume there is a $\eps_0\in (0,D)$ so that there is a divergent sequence $p_i\to\infty$ with $|t_i-t_i'|\ge \eps_0>0$. Here, we have simplified notations in the way that $t_i\defeq t_{p_i}$ and $t_i'=t_{p_i}'$. We also have $|t_i-t_i'|\le D$ from \eqref{eq:ttprimeClose}. So up to taking subsequence, we can assume that $(M,g,\gamma(t'_i))\to (\R\times K, \dist, (0,k_0))$ and $\gamma(t_i)\to (a,k_0)$ with $|a|\in [\eps_0,D]$, granted Proposition \ref{prop:RayConverge}. 
      
      Let $q_i^-=\gamma(0)$, and $q_i^+=\gamma(T_i)$ for some $T_i\to \infty$ to be chosen, we define $$b_i^+(p)=(T_i-t_i)-\dist_g(p, q_i^+),$$ and $\mathbf{b}_i^+$ be the harmonic replacement of $b_i^+$ in $B_{8D+1}(p_i)$. We aim to prove that $p_i\to (0,k)$ for some $k
      \in K$. By the almost splitting theorem, Theorem \ref{thm:AlmostSplitting}, we only need to show $\mathbf{b}_i^+(p_i)\to 0$. Here $\mathbf{b}_i^+(p_i)$ makes sense because $p_i\in B_{2D+1}(\gamma(t_i'))\subset B_{8D+1}(\gamma(t_i'))$. We see that by definition, for any $\eps_i\to 0^+$ there exists $T_i\to\infty$ so that $|b^+_i(p_i)|=|b^+_i(p_i)-(b_\gamma(p_i)-t_i')|<\eps_i$. It follows from Lemma \ref{lem:HarmApprox} that $|\mathbf{b}_i^+(p_i)|\le \Psi_i+\eps_i$, so $\mathbf{b}_i^+(p_i)\to 0$ follows.
      
      Now by the definition of $t_i$ we have $\dist_g(p_i,\gamma(t_i))\le \dist_g(p_i,\gamma(t'_i))$. Let $i\to \infty$, we have 
      $$
      a^2+\dist^2(k,k_0)=\dist^2((0,k),(a,k_0))\le \dist^2((0,k),(0,k_0))=\dist^2(k,k_0).
      $$ 
      This is a contradiction to $|a|\in [\eps_0,D]$, and we conclude the proof. 
\end{proof}

We are at a good stage to generalize the result of \cite{antonelli2023isoperimetric}*{Theorem 3.10} to nonnegative Ricci curvature.

    \begin{proposition}\label{prop:conv}
         Let $(M,g)$ be an open $n$-manifold with $\Ric_g\ge 0$ such that $M$ is noncollapsed \eqref{eq:ncends} 
         and $M$ does not split. Let $\gamma$ be a ray in $M$ and $b_\gamma$ be the associated Busemann function. Suppose there exists $D>0$ such that $\dist_g(p,\gamma)<D$ for every $p\in M$. Take a sequence $t_i\nearrow\infty$ such that $(M,\gamma(t_i))$ converges to $(\R\times K, (0,k_0))$ for some compact $\ncRCD(0,n-1)$ space $K$ and $k_0\in K$. Denote by $P:\R\times K\to \R$ the projection onto the $\R$-factor, the following holds.
         \begin{enumerate}
             \item\label{item:L1conv} $b_{\gamma}-t_i$ locally uniformly converges to $P$ along the pGH convergence $(M,\gamma(t_i))\to (\R\times K, (0,k_0))$. In particular for any $s,s_1,s_2\in \R$ with $s_1<s_2$ the (characteristic functions of) corresponding (sub)level sets converge.
             \begin{align}
                 \{b_\gamma-t_i\le s\}\to \{P\le s\}&,\  \text{in}\ L^1_{\mr{loc}};\\
                 \{s_1\le b_\gamma-t_i\le s_2\}\to \{s_1\le P\le s_2\}&,\ \text{in}\ L^1.
             \end{align}
             \item \label{item:mono} For any $s\in R$, $\{b_\gamma<s\}$ has finite perimeter, and for any $s_1,s_2\in \R$ with $s_1<s_2$, the following monotonicity holds
             \begin{align}
                 \mr{Per}(\{b_\gamma<s_1\})&\le \mr{Per}(\{b_\gamma<s_2\})\label{eq:mono}\\
                 \lim_{s\to\infty} \mr{Per}(\{b_\gamma<s\})&=\haus^{n-1}(K)\label{eq:lim}.
             \end{align}
             Moreover the following rigidity holds. If $\mr{Per}(b_\gamma <\bar s)=\haus^{n-1}(K)$ for some $\bar s$, then there exists $s_0\ge \bar s$ such that  $(\{b_\gamma= s\},\dist_g)$ is isometric to $K$ (with the metric given in its $\ncRCD(0,N-1)$ structure).
             \item\label{item:isovol} Let $\{q_i\}_{i\in\N^+}$ be an arbitrary diverging sequence in $M$, and the pGH limit space of $(M,q_i)$ be $(\R\times  K',(0,k'_0))$, for some compact $\ncRCD(0,n-1)$ space $K'$ and $ k'_0\in K'$, then $\haus^{n-1}(K)=\haus^{n-1}(K')$.
         \end{enumerate}
    \end{proposition}

    The proof of item \eqref{item:L1conv} and item \eqref{item:mono} can be done verbatim as in \cite{antonelli2023isoperimetric}*{Theorem 3.10} given Lemma \ref{lem:uniclose}, so we skip the proof. 

     For the rigidity, it is the same as \cite{antonelli2023isoperimetric}*{Theorem 3.10}. At an intuitive level, If $\mr{Per}(b_\gamma <\bar s)=\haus^{n-1}(K)$ for some $\bar s$, then $b_\gamma$ becomes a harmonic function in $\{b_\gamma>\bar s\}$, then a rigidity theorem of Kasue \cite{KasueRigidity} can be invoked to deduce the rigidity we seek for.
      
     We prove item \eqref{item:isovol}, which also completes the proof of Theorem \ref{thm:splitting}.
    \begin{proof}[proof of item \eqref{item:isovol}]
        Let $r_i=b_\gamma(q_i)$. It follows from the bounded distance to $\gamma$ and Corollary \ref{cor:LevelClose} that $\gamma(r_i)\in B_{2D}(p_i)$, so by Theorem \ref{thm:AlmostSplitting}, $(M,g,\gamma(r_i))\to (\R\times K',\dist, x')$ for some $x\in \R\times K'$ possibly different from $(0,k_0')$. We apply \eqref{item:mono} to obtain that $$\haus^{n-1}(K')=  \lim_{s\to\infty} \mr{Per}(\{b_\gamma<s\})=\haus^{n-1}(K).$$  
    \end{proof}

We are finally in position to achieve the goal of this section. 
    
    \begin{proof}[Proof of Theorem \ref{thm:equi}]
    We show $\eqref{item:linearvol}\Rightarrow\eqref{item:uniclose1}\Rightarrow\eqref{item:uniclose2}\Rightarrow\eqref{item:linearvol}$ and $\eqref{item:linearvol}\Rightarrow\eqref{item:splitting2}\Rightarrow\eqref{item:splitting1}
    \Rightarrow\eqref{item:uniclose2}$.
        
        \eqref{item:linearvol}$\Rightarrow$ \eqref{item:uniclose1} is Proposition \ref{prop:bddDist}.

        \eqref{item:uniclose1}$\Rightarrow$ \eqref{item:uniclose2} is trivial.

        \eqref{item:uniclose2}$\Rightarrow$ \eqref{item:linearvol}  We infer from Lemma \ref{lem:NoSplitting} that the dichotomy in \eqref{item:uniclose2} are mutually exclusive. Assume $M$ does not split and
        \begin{align}
             m&\defeq \inf_{x\in M}b_\gamma(x)>-D.
        \end{align}
        Take $x\in B_R(\gamma(0))$. By definition $b_\gamma(x)\le \dist_g(x,\gamma(0)) \le R$, so $x\in \{b_\gamma\le R\}$.
        We invoke also item \eqref{item:mono} of Proposition \ref{prop:conv} and the coarea formula to estimate that
        \begin{align}\label{eq:BallLevel}
            \volg(B_R(\gamma(0)))&\le \haus^n(\{b_\gamma\le R\}) \notag\\
            &\le\int_m^{R}\mr{Per}(\{b_\gamma< r\})\di r= \lim_{r\to\infty}\mr{Per}(\{b_\gamma< r\}) (R+D).
        \end{align}
         The limit space at infinity taken along $\gamma$ splits as $\R\times K$ by Theorem \ref{thm:AlmostSplitting}. We deduce from Lemma \ref{lem:limitBdd} that $K$ is compact hence has finite $\haus^{n-1}$ volume because $K$ is $\ncRCD(0,n-1)$. In particular, $\lim_{r\to\infty}\mr{Per}(\{b_\gamma< r\})=\haus^{n-1}(K)$ is finite.

        $\eqref{item:linearvol}\Rightarrow \eqref{item:splitting2}$ is Theorem \ref{thm:splitting}.

        $\eqref{item:splitting2} \Rightarrow \eqref{item:splitting1}$ is trivial.

         $\eqref{item:splitting1} \Rightarrow \eqref{item:uniclose2}$  
         If $M$ splits off $\R$ this is clear. Assume $M$ does not split. We claim that the $\gamma$ provided by \eqref{item:splitting1} is the ray we search for and argue by contradiction. 
         It follows from the proof of Proposition \ref{prop:bddDist} that there exists a ray $\sigma$ with 
         \begin{equation}\label{eq:perp2}
            \sigma(0)=\gamma(0),\ \dist_g(\sigma(s),\gamma)=\dist_g(\sigma(s),\gamma(0))=s, \ \forall s\ge 0.
        \end{equation}
The first item there is excluded because of our assumption. 

We claim that $b_\gamma(\sigma(s))$ is bounded from above. If not, there exists $s_i$ such that $t_i\defeq b_\gamma(\sigma(s_i))\to \infty$. Then it is easily seen by \eqref{eq:perp2} that $s_i=\dist_g(\sigma(s_i),\gamma(0))\ge b_\gamma(\sigma(s_i))=t_i\to \infty$. On the other hand, up to subsequence $(M,\gamma(t_i))$ converges to $\R\times K$ and $K$ is compact. Proposition \ref{prop:LevelDiam} shows that $\limsup_{i\to \infty}\mr{diam}(b_\gamma^{-1}(t_i))<\infty $. However, this would force $s_i\le \dist_g(\gamma(t_i),\sigma(s_i))\le \mr{diam}(b_\gamma^{-1}(t_i))<\infty$, a contradiction. The claim is proved. Let us assume $\sigma\subset \{b_\gamma\le T\}$ for some $T$. There is a $t_j$ such that $t_j>T+1$. The proof of Lemma \ref{lem:bddBuse} can be applied to show that $\{t_j-1\le b_\gamma\le t_j\}$ is compact. Now $M\setminus \{t_i-1\le b_\gamma\le t_i\}$ has at least two noncompact connected components contained in $\{b_\gamma\le T\}$ and $\{b_\gamma >t_i\}$ because both of them contain a ray, which means $M$ has at least $2$ ends. So $M$ splits, a contradiction.
    \end{proof}

\section{Application to isoperimetric problem}\label{sec:iso}
   The results in previous section clear an obstacle of studying the existence of isoperimetric sets of a given large volume. Therefore we are able to generalize \cite{antonelli2023isoperimetric}*{Theorem 5.13} to nonnegative Ricci curvature. Except for the fact that the limit space at infinity is unique. 
   Note that for isoperimetric problem we only consider noncollapsed spaces because of \cite{isoregion}*{Proposition 2.18}. It is worth pointing out that most of the technical work has been done in \cite{antonelli2023isoperimetric}.
    
    Let us first briefly recall the problem without the attempt to exhaust literature in this field of study. Let $(M,g)$ be an open $n$-manifold with $\Ric_g\ge 0$. The isoperimetric problem asks for a given volume $V>0$, the existence of a set of finite perimeter that realizes 
    \begin{equation*}
    	\inf\{\mr{Per}(E): E\subset M\ \text{Borel}, \haus^n(E)=V\}\defeq I_M(V).
    \end{equation*}
   We call $I_M:[0,\volg(M)]\to \R_{\ge 0}$ the isoperimetric profile of $M$. A set of finite perimeter $E$ such that $\mr{Per}(E)=I_M(V)$ is called an isoperimetric set. One of the difficulties in establishing the existence results for isoperimetric sets is that a minimizing sequence may escape to infinity. It is now understood that an asymptotic mass decomposition technique can be applied to $\RCD(K,N)$ spaces to study how a minimizing sequence can escape to infinity, see \cite{antonelli2022isocomparison}*{Theorem 4.1}, \cite{antonelli2023isoperimetric}*{Theorem 2.22} and references therein. For $K=0$, this technique can be pushed further \cite{antonelli2022isocomparison}, deriving
   
   \begin{proposition}[\cite{antonelli2023isoperimetric}*{Proposition 2.25}]\label{prop:asym}
   	Let $(M,g)$ be an open $n$-manifold with $\Ric_g\ge 0$ and noncollapsed ends \eqref{eq:ncends}. Given $V>0$, let $\{E_i\}_{i\in \N}$ be a minimizing sequence for $I(V)$ with bounded perimeters. Then exactly one of the following happens  
   	\begin{enumerate}
   		\item \label{item:stay} either there exists an isoperimetric set $E$ for $I_M(V)$,
   		\item \label{item:escape} or there exist a limit space at infinity $(X,\dist,\meas)$ and an isoperimetric set $E\subset X$ so that $I_M(V)=\mr{Per}(E)$.
   	\end{enumerate}
   
   \end{proposition}
     
   Therefore, the focus is to understand the isoperimetric sets in the limit space at infinity. It seems that this is even harder than studying it on manifolds. However, splitting drastically simplifies the structure of a limit space at infinity. This is most powerful when $M$ has linear volume growth. As shown in Theorem \ref{thm:splitting}, in this case any limit space at infinity is a cylinder $\R\times K$, and the isoperimetric sets are well understood in such cylinders by \cite{antonelli2023isoperimetric}*{Appendix A}. 
   
   We set $\mathfrak{D}\defeq \haus^{n-1}(K)$ for any limit space at infinity $\R\times K$. It is justified by Theorem \ref{thm:splitting} that all $K$'s have the same $\haus^{n-1}$ volume. The following is a generalization of \cite{antonelli2023isoperimetric}*{Theorem 5.13} to Ricci curvature lower bound.
   
   \begin{theorem}\label{thm:AP}
   	Let $(M,g)$ be an open $n$-manifold with $\Ric_g\ge 0$ such that $M$ is noncollapsed \eqref{eq:ncends}, has linear volume growth \eqref{eq:linearvol}, and $M$ does not split.  
   	Then the following holds.
   	\begin{enumerate}
   		\item \label{item:exist} There exists $V_0\defeq V_0(v,n,\mathfrak{D})$, such that an isoperimetric set exists for any $V\ge V_0$. (recall that $v$ is from \eqref{eq:ncends}.)
   		\item \label{item:rigidity} For every $V>0$, we have $I(V)\le \mathfrak{D}$. If for some $V'>0$, the equality $I(V')= \mathfrak{D}$ holds, then it holds for all $V\ge V'$ and there exists an isoperimetric set $\Omega$ of volume $\haus^{n-1}(\Omega)\ge V'$ such that $(\partial \Omega,\dist_g|_{\partial \Omega})$ is a $\ncRCD(0,n-1)$ space and that $(M\setminus \Omega,\dist_g)$ is isometric to $([0,\infty)\times \partial \Omega, \dist_g|_{\partial \Omega}\otimes\dist_{\mr{eu}})$.
   		\item \label{item:global} Let $E_i$ be a sequence of isoperimetric sets with $\haus^{n}(E_i)\to \infty$, $x_i\in \partial E_i$. Let $(\R\times K,\dist_{\mr{eu}}\otimes \dist_K, (0,k_0))$ be the pGH limit space of $(M,g, x_i)$. Then $E_i$ converges to $(-\infty,0)\times K$ in $L^1_{\mr{loc}}$ and $\partial E_i$ converges to $\{0\}\times K$ along the pGH convergence. Moreover there exists a sequence of isoperimetric sets $\Omega_i\Subset\Omega_{i+1}$ so that $\cup_i \Omega_i=M$ and 
   		\[\lim_{V\to\infty}I_M(V)=\mathfrak{D}.\]
   	\end{enumerate}
   \end{theorem}
   Again, the proof will be almost the same as the original one \cite{antonelli2023isoperimetric}*{Theorem 5.13}. Instead of reproducing the proof, we only point out how results in this note are applied. We refer the readers to the aforementioned paper for a detailed proof. 
   
   For item \eqref{item:exist}. The only subtlety compared to the original proof is the loss of uniqueness of the limit spaces at infinity, but this is resolved by Theorem \ref{thm:splitting} which says even different limit spaces at infinity have the same volume. One shows that for a minimizing sequence it cannot be the case \eqref{item:escape} in Proposition \ref{prop:asym}. If this was the case, then one can use sublevel sets $\{b_\gamma\le s\}$ as a competitor for $I_M(V)$. 
   
   For item \eqref{item:rigidity}. Nothing need to be changed. This is the Kasue-type rigidity in Proposition \ref{prop:conv} item \eqref{item:mono}.
   
   For item \eqref{item:global}. With Theorem \ref{thm:splitting}, Lemma \ref{lem:uniclose} and Proposition \ref{prop:conv}, \eqref{item:L1conv}. The proof can also be done verbatim. 

  The following consequence is pointed out by the referee.
   \begin{corollary}
       Let $(M,g)$ be an open $n$-manifold with $\Ric_g\ge 0$ such that $M$ is noncollapsed \eqref{eq:ncends}, has linear growth \eqref{eq:linearvol}, and $M$ does not split. Fix a base point $o\in M$, then 
       \[
       \limsup_{r\to\infty} \frac{\volg(B_r(o))}{r}=\mathfrak{D}=\lim_{V\to\infty} I(V).
       \]
   \end{corollary}

   \begin{proof}
       Take a ray $\gamma$ emanating from $o$. We let $D>0$ be such that $\dist_g(p,\gamma)<D$ guaranteed by Proposition \ref{prop:bddDist}. Then the level sets of $b_\gamma$ have a uniform diameter upper bound $4D$ and $m\defeq \inf_M b_\gamma>-D$ thanks to Lemma \ref{lem:NoSplitting}. Recall from the proof of Theorem \ref{thm:equi}, specifically \eqref{eq:BallLevel}. We obtained there that for sufficiently large $r>0$, $B_r(o)\subset \{b_\gamma\le r\}$ and 
       \[
       \limsup_{r\to\infty} \frac{\volg(B_r(o))}{r}\le \mathfrak{D}.
       \]
       To show the reverse inequality, we see from the uniform diameter upper bound that $\{b_\gamma\le r\}\subset B_{r+4D}(o)$. By co-area formula we have 
       \[
       \lim_{r\to\infty} \frac 1 r\int_m^r \mr{Per}(\{b_\gamma<t\})\di t\le  \limsup_{r\to\infty} \frac{\volg(B_{r+4D}(o))}{r},
       \]
       it follows that 
       \[
       \mathfrak{D}\le \limsup_{r\to\infty} \frac{\volg(B_r(o))}{r}.
       \]
      This completes the proof since $\lim_{V\to\infty}I_M(V)=\mathfrak{D}$ is contained in item \eqref{item:global} of Theorem \ref{thm:AP}.
   \end{proof}
   
   \begin{remark}\label{rmk:isoprofile}
   	We are unfortunately unable to add the statement 
   	\begin{equation*}
   		\text{There exists $C>0$ such that for any $V>0$, $I_M(V)<C$.}
   	\end{equation*}
   	to the list of equivalence conditions in Theorem \ref{thm:equi}. The argument in \cite{antonelli2023isoperimetric}*{Corollary 3.11 (4) $\Rightarrow$ (1) or (3)} can also be applied in our setting to show that this statement implies that there exists a diverging sequence (in fact, many such sequences exist) for which the limit space at infinity splits off a line. However, to go from one sequence to all sequences seems to be out of reach. Because different sequence may have different ``escape speed'', i.e., the growth of $\dist(x,p_i)$ for fixed $x$ and $p_i\to\infty$. For Alexandrov spaces, the escape speed does not matter and it is shown in \cite{antonelli2023isoperimetric}*{Theorem 5.13(1)} that the limit space at infinity is unique. For Ricci limit spaces, we do not know if two divergent sequences of different escape speed are related to each other or not. 
   \end{remark}

Finally, we mention that large volume is a necessary condition for the existence of isoperimetric sets. A counterexample \cite{antonelli2023nonexistence} has been found when the volume is small.

    \section{Examples}\label{sec:example}
    In this section we will borrow the examples from \cite{CN11}. Recall that topologically $\C P^2\sharp\overline{\C P^2}$ is the connected sum of $\C P^2$ and another $\C P^2$ with reversed orientation. The examples we will construct are all noncollapsed, since all the limit spaces at infinity have the same dimension as the manifold we start with.  
    
    \begin{theorem}\label{thm:nonhomeo}
    	There exists an open $5$-manifold $M$ with nonnegative Ricci curvature and linear volume growth such that both $\R\times \mb S^4$ and $\R\times \C P^2\sharp\overline{\C P^2}$ are possible limit spaces at infinity of $M$. In particular the cross sections are not homeomorphic. 
    \end{theorem}
    
    \begin{proof}
    	It is constructed in \cite{CN11}*{Example II} a family of metrics $(g_t)_{t\in(0,2]}$ on $\C P^2\sharp\overline{\C P^2}$, so that
    	\begin{enumerate}
    		\item $\Ric_{g_t}\ge (4-2)g_t= 2 g_t$.
    		\item $\frac\di {\di t} \mr{vol}_{g_t}=0$, i.e., the volume are the same for $t\in (0,2]$.
    		\item\label{item:derivative} $|\partial_s g(s)|, |\partial_s\partial_s g(s)|,|\nabla \partial_s g(s)|\le 1$.
    	\end{enumerate} 
    	 Moreover $(\C P^2\sharp\overline{\C P^2},g_t)$ converges in GH topology to $(\mb S^4,\dist_0)$ with $2$ singular points. Based on this model, we construct a cylinder whose cross sections oscillate between $(\C P^2\sharp\overline{\C P^2}, g_2)$ and $(\C P^2\sharp\overline{\C P^2}, g_\eps)$ for $\eps\to 0^+$. To this end, we construct a Riemannian metric on $[0,\infty)\times \C P^2\sharp\overline{\C P^2}$ of the form 
    	 \[
    	 \di r^2+ h^2(r)g_{f(r)}
    	 \]
    	 \textbf{Step 1.} We start with the case when $r$ is large. 
    	
    	Define $f(r)\defeq ({1-e^{-r}})\sin({\frac12\ln \ln r})+1\in (0,2)$. clearly, every real number in $[0,2]$ is a possible limit of $f(r)$ as $r\to \infty$. The Ricci tensors are computed in \cite{CN11}*{Lemma 2.1}. 
    	We use subscript $r$ to represent vector fields along $r$ direction, $i,j$ to represent vector fields along $\C P^2\sharp\overline{\C P^2}$. 
      In what follows we already incorporated the fact that $g^{ij}\dot g_{ij}= \mr{tr}(\dot g)=\frac{\di }{\di t}\mr{vol}_{g_t}=0$.
    	\begin{align}
    		\label{eq:ricrr}\Ric_{rr} &= -4 \frac{h''}{h} -\frac{1}{4}g^{ij}\dot g_{ja}g^{ab}\dot g_{bj}.\\
    		\label{eq:ricir}\Ric_{ir} &=\frac12 \left[\partial_a(g^{ab}\dot g_{bj})+\frac12 \dot g^{qb}\left(\partial_i g_{ab}-g_{ib}g^{pq}\partial_a(g_{pq})\right)\right]\ge -\frac{3}{2}|f'|.\\
    		\Ric_{ij} &=\Ric_{ij}^{g_t}+h^2\left[\left(-\frac{3h'^2}{h^2}-\frac{h''}{h}\right)g_{ij}-\frac{5h'}{2h}\dot g_{ij}+\frac12 g^{ab}\dot g_{ai}\dot g_{bj}\right].
    	\end{align}
    	Let $h(r)={D\arctan (\ln r)}$ for $D>0$ to be determined. 
    	We can easily verify the linear volume growth. Now let us verify the nonnegative Ricci curvature. By our definition of $f(r)$ we have $|f'(r)|\le \frac 1{2r\ln r}$ and 
    	\begin{align}\label{ieq:ricrr}
    		\Ric_{rr} &= -4 \frac{h''}{h} -\frac{1}{4}g^{ij}\dot g_{ja}g^{ab}\dot g_{bj}\notag\\
    		&\ge 4\frac{(\ln r+1)^2}{r^2 (\ln ^2 r+1)^2\arctan(\ln r)}-\frac{1}{16 r^2\ln^2 r}-e^{-2r}-2e^{-r}\ge \frac 8\pi \frac{\ln^2 r}{r^2(\ln^2 r+1)^2}-\frac{1}{16r^2\ln^2 r}\notag\\
    		&\ge \frac{9}{4r^2\ln^2 r}\ge 0,\notag
    	\end{align}
       provided that $r\ge e$ and $ \frac 8\pi \frac{2\ln r+1}{r^2(\ln^2 r+1)^2}\ge e^{-2r}+2e^{-r}$. 
    
    Next note that 
    \begin{equation*}
    	\text{ $\frac{h'}{h}, \frac{h''}{h}\to 0$ as $r\to \infty$, and $h$ is bounded.}
    \end{equation*} 
    We have the following control 
    \begin{equation}\label{ieq:ricij}
    	\Ric_{ij}\ge 2g_{ij}- o\left(\frac{1}{r^2}\right)D^2g_{ij}-\frac{5h'h}{2}\dot g_{ij}+\frac{h^2}2 g^{ab}\dot g_{ai}\dot g_{bj}.
    \end{equation}
    Here the last two terms are zero when $i\neq j$ for $g_t$, and when $i=j$ it also decays faster than $\frac1 {r^2}$ by our choice of $f(r)$. In conclusion we have $\Ric_{ij}\ge g_{ij}$ for large $r$.
   Finally, we are left with the verification of $\Ric_{ir}\ge 0$. First we see from \eqref{eq:ricir} that
     \begin{equation}\label{ieq:ricir}
     	\Ric_{ir}\ge -\frac3{4r\ln r}-e^{-r}\ge -\frac3{2r\ln r}.
     \end{equation}
      Then we adapt the strategy of \cite{CN11}*{Lemma 2.1} to show the positivity of $\Ric_{ir}$. Fix a point and find an orthonormal basis at this point. Write every possible mixed direction as $\delta\hat r+\sqrt{1-\delta^2}\frac{\hat i}{h}$ for $\delta\in [0,1]$. It holds
      \[
      \Ric_{(\delta \hat r+\frac{\sqrt{1-\delta^2}\hat i}{h})(\delta \hat r+\frac{\sqrt{1-\delta^2}\hat i}{h})}\ge \frac{9\delta^2}{4r^2\ln^2 r} - \frac{3\delta\sqrt{1-\delta^2}}{h r\ln r}+\frac{1-\delta^2}{h^2}\ge \left(\frac{3\delta}{2r\ln r}-\frac{\sqrt{1-\delta^2}}{h}\right)^2\ge 0.
      \]
      
       Choose $r_k$ to be an increasing sequence so that $\sin (\frac12 \ln\ln r_k)=-1$ and choose $r_k'$ to be an increasing sequence so that $\sin (\frac12 \ln\ln r_k)=1$, and fix an arbitrary point $p\in \C P^2\sharp\overline{\C P^2}$. 
       It is immediate that the pGH limit along $(r_k,p)$ is $(\R\times \mb S^4,
       \dist_{\mr{eu}}\otimes \dist_0$), 
       and the pGH limit along $(r_k',p)$ is $(\R\times  \C P^2\sharp\overline{\C P^2}, 
     \dist_{\mr{eu}}\otimes \dist_{g_1})$.
     
       \textbf{Step 2.} This step is to transit the Riemannian metric on $\C P^2\sharp\overline{\C P^2}$ from some $g_{f(r)}$ to $g_2$, which is known to be closable. Assume that the previous construction is done in $[\bar r,\infty)\times  \C P^2\sharp\overline{\C P^2}$ for some $\bar r>e$ large and satisfying $\sin (\frac12\ln\ln \bar r)= 1$, i.e., $\bar r$ will be chosen among $e^{e^{(4k+1)\pi}}\defeq r_{k}$, $k\in \N^+$. We extend our construction to $[\bar r_{k-1} ,\bar r_k]$ with the following properties.
       \begin{itemize}
     	\item $f$ can be smoothly extended on $[\bar r_{k-1},\bar r_k]$ so that $ f(\bar r_{k-1})=2$, $ f'(\bar r_{k-1})=0$ and $f(r)\in (0,2]$ for $r\in [\bar r_{k-1},\bar r_k]$. In particular $\Ric_{ir}=0$ at $\bar r_{k-1}$ and $\dot g_{ij}$ has no contribution to the second fundamental form at $\bar r_{k-1}$.      	
     	\item $\Ric_{rr}$ and $\Ric_{ij}$ remains positive for any $r\in [\bar r_{k-1},\bar r_k]$. 
     \end{itemize}
     To achieve the two items above, let $\psi(r)$ be a smooth function on $\R$ such that when $r\le 0$, $\psi(r)=1$, when $r\ge1$, $\psi(r)=0$ and $\psi(r)\in[0,1]$ on $\R$ and its derivative of all orders vanish at $0,1$. Set 
     $$
     e(r)=1-e^{-r}+\psi\left(\frac{r-r_{k-1}}{r_k-r_{k-1}}\right)e^{-r}.
     $$ 
     On $[\bar r_{k-1}, \bar r_k]$, we define $f(r)= e(r)\sin (\ln\ln r)$. Then it is readily checked that $f$ is smooth at $r_k$, $f(\bar r_{k-1})=2$, and $f'(\bar r_{k-1})=0$.  Now we check the Ricci tensors. First by \eqref{eq:ricrr} we have
     \[
     \Ric_{rr}\ge \frac8{\pi\bar r^2\ln ^2 \bar r}-\frac{1}{4}|f'|^2\ge \frac9{4\bar r^2\ln ^2 \bar r}.
     \]
      This is still true provided that $\bar r=\bar r_k$ is large enough since our modification only contributes some terms with exponential decay. For $\Ric_{ij}$ the same reasoning works, we just need to choose $\bar r$ large so that $\Ric_{ij}\ge g_{ij}$. It again remains to see $\Ric_{ir}$. Recall \eqref{eq:ricir}, we can again use the same trick as the one right below \eqref{ieq:ricir}. 
      
      We now compute the second fundamental form $\mr{II}$ at $\bar r_{k-1}$. Since $f'=0$, we have that 
     \begin{align}
        \mr{II}_{ij}&= -\frac{h'}{h}\partial_r g_{ij}= -\frac{\partial r}{\tilde r(1+\ln ^2 \tilde r)\arctan(\ln \tilde r)}g_{ij},\notag
     \end{align}
     \textbf{Step 3.} For the last step, we would like to glue the manifold with boundary we constructed to the so-called $\mc{C}_2$ constructed in \cite{CN11}*{Example II}. This would give us a manifold without boundary. 
     For the gluing, note that it is shown that the second fundamental form of $\mc{C}_2$ is strictly positive, say $|\mr{II}^{\mc{C}_2}|> \lambda>0$. 
     By choosing $\bar r_k$ big enough, we can have $|\mr{II}_{ij}|<\lambda$. 
     
     So far we have already chosen $\bar r$ by combining all the requirements in the last two steps. 
     It remains to determine the constant $D$ in $h(r)$.  We want $\{\bar r_{k-1}\}\times \C P^2\sharp\overline{\C P^2}$ to be isometric to $\partial \mc{C}_2$. This can be achieved by choosing some $D\in (0,1)$ since ${\arctan(\ln r)}\to \frac\pi 2$ and $\mc{C}_2$ is fixed and $\partial \mc{C}_2$ has small diameter. This means the previous bound for $\Ric_{ij}$ still holds so the bounds for $\Ric_{rr}$ and $\Ric_{ir}$ also hold.  
     In conclusion after scaling the nonnegativity of the Ricci tensors of $[\bar r,\infty)\times \C P^2\sharp\overline{\C P^2}$ 
     equipped with $\di r^2+ h^2(r)g_{f(r)}$ is preserved. 
     Now by Perelman \cite{PerelmanLargeBetti}, we can glue $[\bar r,\infty)\times \C P^2\sharp\overline{\C P^2}$  and $\mc{C}_2$ along the boundary, so we are done with our construction.
    \end{proof}

    We proceed to the next example. As suggested by Daniele Semola, we can also construct a manifold with nonisometric limit spaces at infinity in dimension $3$, reducing the dimension requirement in Sormani's example \cite{SormaniMiniVol}*{Example 27}, which is $4$.
    \begin{theorem}\label{thm:3dnoniso}
	  For a given volume $V>0$, there exists an open $3$-manifold with nonnegative Ricci curvature and linear volume growth such that the cross sections of limit spaces at infinity can be $\mb S^2$ with infinitely many different metrics of the form $d r^2+\sin^2(ar)g_{\mb S^1}$ for $a\in (0,1)$ and $r\in [0,\pi]$. Moreover with any of these metrics, the $\haus^2$ volume of $\mb S^2$ is $V$.	
	\end{theorem}
	
	\begin{proof}
		It follows from \cite{CN11}*{Example I} and our previous construction. For a parameter $t>0$, we denote the spherical suspension over a manifold $(M,g)$ by $S_t(M)$ with Riemannian metric $\di r^2+ \sin^2 (\frac rt)g$ for $r\in (0,t\pi)$. Consider $S_{t_1}(\mb S^1_{t_0})$ with parameters $t_0,t_1\in (0,1)$, where $t_0$ is the radius of $\mb S^1_{t_0}$. We normalize the given $V$ to be $\haus^2(S_{1/2}(\mb S^1_{1/2}))$, as the construction is the same for all other volumes. The set 
		\[
		\Omega\defeq\{(t_0,t_1)\in (0,1)^2: 0< t_0\le t_1<1, \haus^2 (S_{t_1}(\mb S^1_{t_0}))=\haus^2(S_{1/2}(\mb S^1_{1/2}))\}
		\]
		is a smooth curve in $(0,1)^2$ that contains $(1/2,1/2)$. For the construction we can fix a parametrization so that $\Omega=(0,1)$ and $(1/2,1/2)$ corresponds to $\frac12$. There are several ways to smooth $S_{t_1}(\mb S^1_{t_0})$ while keeping positive Ricci curvature, for example removing a neighborhood of each singularity and gluing in a spherical cap, which is possible by Perelman \cite{PerelmanLargeBetti}. For convenience we use Ricci flow starting from a $2$-dimensional Alexandrov space \cite{RichardSmoothing} instead, since it naturally gives rise to a one-parameter family of metrics with the same lower curvature bound and a uniform control of distances hence volumes \cite{RichardSmoothing}*{Theorem 1.1}. Let $\dist_t$ be the possibly singular metric corresponding to $t\in(0,1)=\Omega$, and $(g_{t,s})_{s\in(0,1]}$ be the Ricci flow starting from $\dist_t$. Along with our choice of parametrization of $\Omega$, we can also assume the $t$ derivatives of $g_{t,s}$ satisfies \eqref{item:derivative} of Theorem \ref{thm:nonhomeo}.  
		 We consider the metric on $[\bar r,\infty)\times \mb S^2$ of the form
		\begin{equation*}
		    \di r^2+ h^2(r)g_{f(r),l(r)}
		\end{equation*}
		Again take $h(r)=\arctan (\ln r)$, $f(r)=\frac{1-e^{-10 r}}2\sin (\frac12\ln \ln r)+\frac12\in (0,1)$ and $l(r)=1- e^{-\frac1{\ln \ln \ln r}}$. It is easy to see that as $r\to \infty$ the possible limits of $f(r)$ are exactly points in $[0,1]$. The construction is now the same as the previous one. To see the bounds of $s$ derivatives as in \eqref{item:derivative} of Theorem \ref{thm:nonhomeo}, we recall classical curvature estimates along Ricci flow which in turn implies that
		\begin{align*}
			 |\partial_s g_{t,l(s)}|&=2\left|Ric^{g_{t,l(s)}}l'(s)\right|\le\frac{C_1e^{-\frac1{\ln \ln \ln s}}}{s\ln s \ln \ln s(\ln\ln\ln s)^2(1-e^{-\frac1{\ln \ln \ln s}})}=o\left(\frac1{s\ln s}\right),\\ 
			|\partial_s \nabla g_{t,l(s)}|&=2\left|\nabla Ric^{g_{t,l(s)}}l'(s)\right|\le\frac{C_2e^{-\frac1{\ln \ln \ln s}}}{s \ln s\ln \ln s (\ln\ln\ln s)^2(1-e^{-\frac1{\ln \ln \ln s}})^2}=o\left(\frac1{s\ln s}\right),\\
			|\partial_s\partial_s g_{t,l(s)}|&=2|\partial_s \Ric^{g_{t,l(s)}}l'(s)+\Ric^{g_{t,l(s)}}l''(s)|\\ \ &\le \frac{C_3e^{-\frac2{\ln \ln \ln s}}}{(s \ln s\ln \ln s)^2 (\ln\ln\ln s)^4(1-e^{-\frac1{\ln \ln \ln s}})^3}+O\left(\frac1{s^2\ln s}\right)=o\left(\frac1{s\ln s}\right).
		\end{align*}
		Here we have used the evolution equation of Ricci curvature under Ricci flow and $\partial_s$ denotes the derivative w.r.t. the second subscript of $g_{t,s}$. 
		
		 As in \textbf{Step 2} of the previous example, after completing the contruction for large $r$, we let the metric to transit from $g_{f(\bar r), e^{-\bar r}}$ to $g_{\frac12,0}=g_{\frac12}$ and then we can glue in a disk to close it.  
		
	\end{proof}
    \bibliographystyle{amsalpha}
\bibliography{lineandend}
\end{document}